\documentclass[11pt]{amsart}
\usepackage[utf8]{inputenc}  
\usepackage[T1]{fontenc}
\usepackage[english]{babel}
\usepackage{amsmath,textcomp,amssymb,geometry,graphicx,enumerate,multicol}
\usepackage{multirow}
\usepackage{algorithm} 
\usepackage[noend]{algpseudocode} 
\usepackage{hyperref}
\usepackage{stmaryrd}
\usepackage{subcaption}
\usepackage{ytableau}

\DeclareMathOperator{\col}{col}

\newcommand{\thickhline}{%
    \noalign {\ifnum 0=`}\fi \hrule height 1.25pt
    \futurelet \reserved@a \@xhline
}

\usepackage{thm-restate}

\newtheorem{theorem}{Theorem}[section]
\newtheorem{lemma}[theorem]{Lemma}
\newtheorem{proposition}[theorem]{Proposition}
\newtheorem{corollary}[theorem]{Corollary}

  \newtheorem{remark}[theorem]{Remark}
  \newtheorem{definition}[theorem]{Definition}
  \newtheorem{example}[theorem]{Example}
\numberwithin{equation}{section}

\hypersetup{
    colorlinks=true,
    linkcolor=blue,
    filecolor=magenta,      
    urlcolor=blue,
}

\newcommand{\la}{\lambda}
\newcommand{\lam}{\lambda}
\newcommand{\lm}{\lambda/\mu}
\usepackage{amssymb}
\usepackage{youngtab}
\newcommand{\SSYT}{\mathrm{SSYT}}

\newcommand{\RPP}{\mathrm{RPP}}
\newcommand{\E}{\mathcal{E}}

\newcommand{\Br}{\mathrm{Br}}
\newcommand{\z}{\mathbf{z}}
\newcommand{\NI}{\mathcal{NI}}
\newcommand{\nOOT}{\mathrm{OOT}}
\newcommand{\cOOT}{\mathcal{OOT}}
\newcommand{\NIP}{\mathcal{NIP}}
\newcommand{\dd}{\mathsf{d}}
\newcommand{\ED}{\mathcal{E}}
\newcommand{\nED}{\mathrm{ED}}
\newcommand{\LR}{\mathrm{LR}}
\newcommand\qbinom[2]{{\begin{bmatrix} #1\\ #2 \end{bmatrix}_q}}

\DeclareMathOperator{\tr}{tr}

\newcommand{\HG}{HG}

\title[Minimal skew semistandard tableaux and the Hillman--Grassl correspondence]{Minimal skew semistandard tableaux and the Hillman--Grassl correspondence}

\author[Morales]{Alejandro H. Morales}
\address[A. H. Morales]{LACIM, Département de Mathématiques, Universit\'e du Qu\'ebec \`a Montr\'eal, Canada} 
\address{Department of Mathematics and Statistics, UMass Amherst, Amherst, U.S.A.}
\email{morales\_borrero.alejandro@uqam.ca}
\urladdr{https://ahmorales.combinatoria.co}

\author[Panova]{Greta Panova}
\address[G. Panova]{Department of Mathematics, University of Southern California, Los Angeles, U.S.A.}
\email{gpanova@usc.edu}
\urladdr{https://sites.google.com/usc.edu/gpanova/home}

\author[Park]{GaYee Park}
\address[G. Park]{LACIM, Universit\'e du Qu\'ebec \`a Montr\'eal, Canada}
\email{park.ga\_yee@uqam.ca}
\urladdr{https://sites.google.com/view/gayeepark/home}

\thanks{A. H. Morales was partially supported by NSF grant DMS-1855536 and DMS-2154019, G. Park was partially supported by NSF grant DMS-1855536, and G. Panova was partially supported by the NSF grant DMS-1939717.}

\date{\today}

\begin{document}

\maketitle
\begin{abstract}
   Standard tableaux of skew shape are fundamental objects in enumerative and algebraic combinatorics and no product formula for the number is known. In 2014, Naruse gave a formula (NHLF) as a positive sum over excited diagrams of products of hook-lengths. Subsequently, Morales, Pak, and Panova gave a $q$-analogue of this formula in terms of skew semistandard tableaux (SSYT). They also showed, partly algebraically, that the Hillman--Grassl map, restricted to skew semistandard tableaux, is behind their $q$-analogue. We study the problem of circumventing the algebraic part and proving the bijection completely combinatorially, which we do for border strips. For a skew shape, we define minimal semistandard Young tableaux, that are in correspondence with excited diagrams via a new description of the Hillman--Grassl bijection and have an analogue of excited moves. Lastly, we relate the minimal skew SSYT with the terms of the Okounkov-Olshanski formula (OOF) for counting standard tableaux of skew shape. Our construction immediately implies that the summands in the NHLF are less than the summands in the OOF and we characterize the shapes where both formulas have the same number of summands.
\end{abstract}


\section{Introduction}

Standard and semistandard tableaux are fundamental objects in
enumerative and algebraic combinatorics. Standard Young Tableaux (SYT) are fillings
of the Young diagram of $\lambda$ with numbers $1,2,\ldots,n$
increasing in the rows and columns. The number of SYTs of shape $\lambda$ is $f^\lambda$, and can also be interpreted as counting linear
extension of a  certain poset associated to $\lambda$. Counting linear
extensions of posets is in general computationally hard; however, the number $f^{\lambda}$ is given by the famous {\em
  hook-length formula} of Frame-Robinson-Thrall \cite{FRT} 1954.

\begin{theorem}[Frame-Robinson-Thrall \cite{FRT}]
For $\lambda$ a partition of $n$ we have  
\begin{equation} \label{eq:hlf}
f^{\lambda} = \frac{n!}{\prod_{(i,j) \in [\lambda]} h(i,j)},
\end{equation}
where $[\lambda]$ is the Young diagram of $\lambda$ and $h(i,j)=\lambda_i-i+\lambda'_j-j+1$ is the {\em hook-length} of the
square $(i,j)$: 
the number of cells directly to the right and directly below $(i,j)$
including $(i,j)$. 
\end{theorem}

The number of standard tableaux of
skew shape $f^{\lambda/\mu}$ gives the dimension of irreducible representations of
affine Hecke algebras \cite{Ram}. Unlike for straight shapes, there is no  product formula for $f^{\lambda/\mu}$. There are determinantal
formulas for $f^{\lambda/\mu}$ like the {\em Jacobi-Trudi identity}. There is also a classical positive formula for $f^{\lambda/\mu}$ involving the {\em Littlewood-Richardson  coefficients} $c^{\lambda}_{\mu,\nu}$. These formulas, however,  are generally not practical for asymptotic estimates.

Central to this paper are two other positive formulas for $f^{\lambda/\mu}$ coming from equivariant Schubert calculus or more explicitly from evaluations of {\em factorial Schur functions}: the Okounkov--Olshanski formula \cite{OO} from 1998 and the Naruse hook-length formula \cite{naruse2014} from 2014. We start with the latter since it resembles \eqref{eq:hlf}.

\begin{theorem}[Naruse \cite{naruse2014}, \cite{MPP1}] \label{thm:IN}
For a skew shape $\lambda/\mu$ of size $n$ we have
\begin{equation} \label{eq:nhlf} \tag{NHLF}
f^{\lambda/\mu}  = \frac{n!}{\prod_{u \in [\lambda]} h(u)} \sum_{D \in \mathcal{E}(\lambda/\mu)} \prod_{u \in D} h(u) \,=\, n! \sum_{D \in \mathcal{E}(\lambda/\mu)}
 \prod_{u\in [\lambda]\backslash D} \frac{1}{h(u)},
\end{equation}
where $\mathcal{E}(\lambda/\mu)$ is
the set of {\em excited diagrams} of $\lambda/\mu$. 
\end{theorem}

The excited diagrams of shape $\lambda/\mu$, denoted by $\mathcal{E}(\lambda/\mu)$, are certain subsets of size $|\mu|$ of the Young diagram of $\lambda$ obtained from the Young diagram of $[\mu]$ by recursively doing local moves.  Excited diagrams are in correspondence with certain semis standard Young tableaux (SSYT) of shape $\mu$ that are {\em flagged}, i.e. with certain bounds on the entries in each row. The NHLF has been actively studied since 2015 by Morales--Pak--Panova \cite{MPP1,MPP2,MPP3,MPP4}, Konvalinka \cite{konvalinka1,konvanlinka2}, Naruse--Okada \cite{NO}  and has applications and extensions in \cite{MR4349569,jiradilok2019roots,park2021naruse}.

The other positive formula by Okounkov-Olshanski is a also a sum over certain SSYT of shape $\mu$ with entries at most $\ell(\lambda)$ called {\em Okounkov--Olshanski tableaux} $\cOOT(\lambda/\mu)$.

\begin{theorem}[{Okounkov--Olshanskii \cite{OO}}]
For a skew shape $\lambda/\mu$ of size $n$ we have
\begin{equation} \label{eq:OO} \tag{OOF}
f^{\lambda/\mu} \, = \, \frac{n!}{\prod_{u\in [\lambda]} h(u)} \,
\sum_{T \in \cOOT(\lambda/\mu)} \, \prod_{(i,j)\in [\mu]} (\lambda_{d+1-T(i,j)} +i-j),
\end{equation}
where $d=\ell(\lambda)$ and $\cOOT(\lm)$ is the set of SSYT $T$ of shape $\mu$ with entries $\leq d$ such that $j-i<\lambda_{d+1-T(i,j)}$ for all $(i,j)\in [\mu]$.
\end{theorem}

In \cite{MZ}, Morales--Zhu did a similar study of \eqref{eq:OO} as Morales--Pak--Panova did for \eqref{eq:nhlf}. In particular, they gave in \cite[Cor. 5.7]{MZ} a reformulation of \eqref{eq:OO} in terms of the following flagged tableaux: SSYT of shape $\lambda/\mu$ with entries in row $i$ are at most $i$ whose set we denote by $\mathcal{SF}(\lambda/\mu)$. These flagged tableaux  correspond to {\em reverse excited diagrams}, certain subsets of size $|\lambda/\mu|$ of the Young diagram {\em shifted shape} $\lambda^*$ similar to excited diagrams.

\begin{corollary}[Flagged tableaux formulation of \eqref{eq:OO}] \label{cor:flag-form}
For a skew shape $\lambda/\mu$ of size $n$ we have
\begin{equation} 
f^{\lm} = \frac{n!}{\prod_{u \in [\lambda]} h(u)} \sum_{T \in \mathcal{SF}(\lambda/\mu)} \prod_{(j, t) \in M(T)} (\lambda_t-t+ i(j,t)-j+1),
\end{equation}
where $M(T)$ is a certain set of labels on the horizontal edges of $T$ between the cells  $(i(j,t),j)$ and $((i(j,t)+1,j)$ so that column $j$ of $T$ has all the numbers from $1$ to $\lambda'_j$ from top to bottom.
\end{corollary}

\subsection{Minimal skew tableaux and new reformulation of \texorpdfstring{\eqref{eq:nhlf}}{(NHLF)}}

A natural question which is the start of our investigation is to  find a reformulation of \eqref{eq:nhlf} in terms of certain SSYT of skew shape $\lambda/\mu$. We define the set $\SSYT_{\min}(\lambda/\mu)$ of {\em minimal skew SSYT}  obtained from the SSYT $T_0$ of shape $\lambda/\mu$ with entries $\{0,1,\ldots,\lambda'_i-\mu'_i-1\}$ by recursively applying certain local moves. Our first result is a bijection $\Phi$ between excited diagrams and tableaux $\SSYT_{\min}(\lambda/\mu)$ that commutes with the respective local moves (Lemma~\ref{lemma: beta and delta}).

\begin{restatable*}[]{theorem}{bijectionEDandSSYTmin}
\label{thm: bijection ED and SSYTmin}
For a skew shape $\lambda/\mu$ the map $\Phi:\E(\lambda/\mu) \to \SSYT_{\min}(\lambda/\mu)$ is a bijection that commutes with the respective local moves, that is $\Phi \circ \beta = \delta \circ \Phi$.
\end{restatable*}

As a corollary we obtain a new reformulation of \eqref{eq:nhlf}. Let $(\theta_1,\ldots,\theta_k)$ be the {\em Lascoux--Pragacz} decomposition of the shape $\lambda/\mu$ into border strips (see Section~\ref{sec:LP decomp}).

\begin{restatable*}[minimal tableaux reformulation of \eqref{eq:nhlf}]{corollary}{reformulationNHLFminSSYT}
\label{cor: reformulation NHLF min tableaux}
For a skew shape $\lambda/\mu$ of size $n$ we have
\begin{equation} \label{eq:reformulate NHLF}
f^{\lambda/\mu} \,=\, \frac{n!}{\prod_{u \in [\lambda]} h(u)} \sum_{T \in \SSYT_{\min}(\lambda/\mu)} \prod_{(i,j)\in [\mu]} h(i+\alpha,j+\alpha),
\end{equation}
where $\alpha$ is the number strips $\theta_k$ such that $\overline{T}(\theta_k(j))>\mu'_j-i$ and $\overline{T}=T-T_0$.
\end{restatable*}

We give an explicit non-recursive description of the tableaux in $\SSYT_{\min}(\lambda/\mu)$  in terms of the Lascoux--Pragacz \cite{LP} decomposition of $\lambda/\mu$ into border strips (see Theorem~\ref{thm: direct char min SSYT}) that immediately shows they are a subset of the skew flagged tableaux $\mathcal{SF}(\lambda/\mu)$ from \eqref{eq:cor:flag-form}. As an application, we obtain that the Naruse formula has fewer terms than the Okounkov--Olshanski formula and characterize the skew shapes where equality is attained. We denote the number of terms of each formula by $\nED(\lambda)$ and $\nOOT(\lambda/\mu)$, respectively.

\begin{restatable*}[]{theorem}{comparenumberterms}
\label{thm: num ED vs OOF}
For a connected skew shape $\lambda/\mu$ with $d=\ell(\lambda)$ and $r=\max\{i\mid \mu_1=\mu_i\}$ we have that $\nED(\lambda/\mu) \leq \nOOT(\lambda/\mu)$ with equality if and only if  $\lambda_d\geq \mu_r+d-r$.
\end{restatable*}

\subsection{Relation with the Hillman--Grassl correspondence}

The hook-length formula \eqref{eq:hlf} has a $q$-analogue by Littlewood that is a special case of Stanley's hook-content formula for the generating series of SSYT of shape $\lambda$.

\begin{theorem}[Littlewood]
For a partition $\lambda$ we have  
\begin{equation} \label{eq:qhlf}
s_{\lambda}(1,q,q^2,\ldots) \,=\, q^{b(\lambda)}\cdot \prod_{u\in \lambda} \frac{1}{1-q^{h(u)}}.
\end{equation}
\end{theorem}

 Hillman and Grassl give a bijective proof of this via a correspondence $HG(\cdot)$ between reverse plane partitions of shape $\lambda$ and arrays of nonnegative integers of shape $\lambda$. The correspondence $HG(\cdot)$ is related to the famous RSK correspondence (see \cite{MPP1}) and has recent connections to {\em quiver representations} \cite{GPT}.

In \cite{MPP1}, Morales--Pak--Panova gave a $q$-analogues of \eqref{eq:nhlf} for the generating functions of SSYT of skew shape $\lambda/\mu$. 

\begin{theorem}[\cite{MPP1}] \label{thm:qNHLF skew}
For a skew shape $\lambda/\mu$ we have 
\begin{equation} \label{eq:qnhlf}
s_{\lambda/\mu}(1,q,q^2,\ldots) \,=\, \sum_{D \in \mathcal{E}(\lambda/\mu)} \prod_{(i,j)\in [\lambda]\setminus D} \frac{q^{\lambda'_j-i}}{1-q^{h(i,j)}}. 
\end{equation}
\end{theorem}

This identity corresponds to restricting the Hillman--Grassl bijection to SSYT of shape $\lambda/\mu$ \cite{MPP1}. The resulting arrays have support on the complement of excited diagrams and with certain forced nonzero entries on {\em broken diagonals}.
In \cite{MPP1} two of the authors with Pak proved equation~\eqref{eq:qnhlf} algebraically, and showed that the inverse Hillman-Grassl map is an injection to SSYT of shape $\lm$. This implied that the restricted Hillman-Grassl map is a bijection between skew SSYT and excited arrays.

The proof of this connection with the Hillman--Grassl map is partially deduced from the algebraic identity and remains mysterious.

Our second main result is to show that the bijection between minimal tableaux and excited diagrams coincides with the  Hillman--Grassl map.

\begin{restatable*}[]{theorem}{PhivsHG}
 \label{thm: phi vs hg}
 The map $\Phi^{-1}$ is equivalent to the Hillman-Grassl map $\HG$ on the minimal SSYT of shape $\lm$. That is, for $T$ in $\SSYT_{\min}(\lm)$ we have 
 \[
 HG(T) = A_D \text{ if an only if } \Phi^{-1}(T)=D,
 \]
 where the array $A_D$ in $\mathcal{A}^*(\lambda/\mu)$ corresponds to $D$ in $\ED(\lm)$; i.e. $\HG^{-1}(A_D) \,=\, \Phi(D)$.
 \end{restatable*}

We obtain a bijective proof of the following identity, between the {\em leading terms} of each summand on the right-hand-side of~\eqref{eq:qnhlf}.

\begin{restatable*}[]{corollary}{leadtermsqNHLF}
 \label{cor:identity leading terms}
 For a skew shape $\lambda/\mu$ we have 
 \begin{equation} \label{eq:lead term qnhlf}
\sum_{T \in \SSYT_{\min}(\lambda/\mu)} q^{|T|} = \sum_{D \in
  \mathcal{E}(\lambda/\mu)}  q^{\sum_{(i,j)\in D} (\lambda'_j-i)}.
 \end{equation}
\end{restatable*} 
Special cases of this identity give $q$-binomial coefficients and $q$-Catalan numbers (see Section~\ref{sec:qnumED}).

We further investigate the Hillman--Grassl bijection restricted to SSYT of shape $\lambda$ and $\lambda/\mu$. We give a fully combinatorial proof of \eqref{eq:qhlf} for border strips (see Theorem~\ref{thm: bijection border strips}).  For straight shapes $\lambda$, we show the following remarkable additivity property of the Hillman--Grassl correspondence. Let $T_{\min}(\lambda)$ be the SSYT of shape $\lambda$ with entries $i-1$ in row $i$.

\begin{restatable*}[]{theorem}{additivityHGstraight}
\label{thm:additivity straight shape}
Let $T$ be in $\RPP(\lambda)$ and let $T_{\min}=T_{\min}(\lambda)$ be the minimum SSYT of shape $\lambda$ 
\[
HG(T) + HG(T_{\min}) = HG(T+T_{\min}).
\]
\end{restatable*}

The additivity fails for the case of skew shapes when $T_{\min}$ is replaced by minimal skew SSYT.

\begin{table}
    \centering
    \begin{tabular}{lll} \hline 

formulation  &  NHLF & OOF \\  \hline
shape $\mu$  &  flagged tableaux of shape $\mu$ \cite{Kre1,MPP1} & Okounkov--Olshanski tableaux \cite{MZ} \\  & excited diagrams \cite{naruse2014}  & \\ 
skew shape $\lm$ &\textcolor{blue}{minimal skew SSYT}&  flagged tableaux of shape $\lm$ \cite{MZ} \\ &  & reverse excited diagrams \cite{MZ}  \\ \hline 
    \end{tabular}
    \caption{The different formulations of the Naruse hook-length formula (NHLF) and the Okounkov--Olshanski formula (OOF).}
    \label{fig:summary}
\end{table}

\section*{Outline}
The paper is divided as follows. In Section~\ref{sec: background},
we give background and definitions. In Section~\ref{sec: minimal SSYT}, we introduce minimal SSYT of shape $\lambda/\mu$ and show the bijection $\Phi$ between the minimal SSYT and the excited diagram of shape $\lm$. We also give a non-recursive definition of the minimal SSYT and the description of the inverse of $\Phi$. In Section~\ref{sec: reformulation NHLF}, we give the reformulation of the Naruse hook formula in terms of minimal SSYT and in Section~\ref{sec: comparison OOF} we give the comparison between the minimal SSYT and Okounkov-Olshanski tableaux. 

In Section~\ref{sec: HG bijection on min SSYT}, we show the relation between the bijection $\Phi$ and the Hillman--Grassl bijection and (Theorem~\ref{thm: phi vs hg}) and give consequence of this result. We give a bijective proof of \eqref{eq:qnhlf} for the case of border-strips in Section~\ref{sec: bijective proof for borderstrips}. In Section~\ref{sec: additivity}, we prove the additivity property of Hillman--Grassl on straight shapes and give counter-examples of such additivity for skew shapes. Finally we end with the final remarks in Section~\ref{sec: final remarks} which includes a comparison between \eqref{eq:nhlf} and \eqref{eq:OO} with the classical formula \eqref{eq:LR} for $f^{\lm}$ involving Littlewood--Richardson coefficients.

\section{Background}\label{sec: background}

\subsection{Young diagrams and tableaux}

Let $\lambda=(\lambda_1,\ldots,\lambda_{\ell})$ denote an integer partition of lengths $\ell=\ell(\lambda)$ and size $|\lambda|=\sum_{i=1}^\ell \lambda_i$. We denote by $[\lambda]:=\{(i,j) \mid 1\leq j \leq \lambda_i, 1\leq i \leq \ell\}$ the {\em Young diagram} of $\lambda$. For $(i,j) \in [\lambda]$, we denote by $h(i,j)=\lambda_i+\lambda'_j -i-j+1$ the {\em hook length} of $(i,j)$ in $[\lambda]$, and $c(i,j)=j-i$ the {\em content} of $(i,j)$. A {\em skew shape} is denoted by $\lambda/\mu$ and its size is $|\lambda/\mu|=|\lambda|-|\mu|$. We assume that $\lambda/\mu$ is connected. Given a skew shape $\lm$ of length $d$, we denote $\lambda^*/\mu^*$ the {\em shifted skew shape} $(\lambda_1+d-1,\lambda_2+d-2,\ldots,\lambda_d)/(\mu_1+d-1,\mu_2+2-2,\ldots,\mu_d)$. We denote the staircase partition by $\delta_n = (n-1,n-2,\ldots,1)$.

A {\em reverse plane partition} of shape $\lambda/\mu$ is an array $\pi$ of shape $\lambda/\mu$ filled with nonnegative integers that is weakly increasing in rows and columns. We denote the set of such plane partitions by $\RPP(\lambda/\mu)$ and the size of a plane partition $\pi$ is the sum of its entries and is denoted by $|\pi|$. A {\em semistandard Young tableaux} of shape $\lambda/\mu$ is a reverse plane partition of shape $\lambda/\mu$ that is strictly increasing in columns. We denote the set of such tableaux by $\SSYT(\lambda/\mu)$. Similarly, $\SSYT(\lm,m)$ denotes the set of SSYT of shape $\lm$ and entries $\leq m$. A {\em standard tableaux} of shape $\lambda/\mu$ of size $n$ is an array $T$ of shape $\lambda/\mu$ with the numbers $1,\ldots,n$, each number appearing once, that is increasing in rows and columns. We denote the number of such tableaux by  $f^{\lambda/\mu}$.

Note that for straight shapes $\lambda$, the generating functions of $\RPP(\lambda)$ and $\SSYT(\lambda)$ are proportional given by the direct bijection $T(i,j) \to T(i,j) -j$:
\begin{equation} \label{eq: gf SSYT RPP lambda}
\sum_{T \in \SSYT(\lambda)} q^{|T|} \,=\, q^{b(\lambda)} \sum_{T \in \RPP(\lambda)} q^{|T|},
\end{equation}
where $b(\lambda)=\sum_i (i-1)\lambda_i$.

\subsection{Schur functions} \label{sec: Schur functions}
Given a skew partition $\lambda/\mu$ and variables ${\bf x} = (x_1, x_2, \ldots)$, let 
\[
s_{\lambda/\mu}({\bf x}) = \sum_{T \in SSYT(\lambda/\mu)} {\bf x}^T,
\]
where ${\bf x}^T = x_1^{\#\text{1s in }(T)}x_2^{\#\text{2s in }T}\cdots$. The generating function for $\SSYT(\lambda/\mu)$ by  size (volume) is obtained by evaluating $s_{\lambda/\mu}$ at $x_i=q^{i-1}$.

\subsection{Lascoux--Pragacz and Kreiman decomposition of \texorpdfstring{$\lm$}{lambda/mu}}
\label{sec:LP decomp}

A \emph{border-strip} is a connected skew shape with no $2 \times 2$ box. 
The starting point and ending point of a strip are its
southwest and northeast endpoints. Given $\lambda$, the {\em outer border strip}  is the  strip
containing all the boxes inside $[\lambda]$ sharing a vertex with the boundary of
$\lambda$., i.e. $\lambda/(\lambda_2-1,\lambda_3-1,\ldots)$.

We introduce two different decompositions of a Young diagram of shape $\lambda/\mu$, the \emph{Lascoux--Pragacz decomposition} and the \emph{Kreiman decomposition}. See Figure~\ref{exLascouxPragacz} for an example. We follow the definitions and notations in \cite{KimYoo}.

For a skew shape $\lm$, let $cont(\lm) = \{c(u) \,|\, u\in \lm\}$ be the set of contents of the cells of $[\lm]$. A \emph{cutting strip} of $\lm$ is a border-strip $\rho$ such that $cont(\rho) = cont(\lm)$. We can decompose $\lm $ into border strips $(\rho_1,\dots, \rho_k)$ by sliding the cutting strip $\rho$ along the diagonal and taking an intersection with $[\lm]$. Denote by $q_i$ the content of the top right-most element in each border strip $\rho_i$. Note that different border strips have different values of $q_i$. We index $\rho_i$ such that $q_1>q_2>\dots>q_k$. See Figure~\ref{fig:cutting edge}. The {\em diagonal distance} of each $\rho_i$ is the number of diagonal slides from the cutting strip $\rho$.

\begin{figure}
    \centering
    \includegraphics[height=4cm]{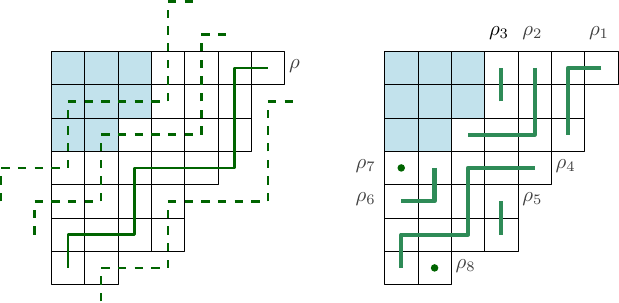}
    \caption{A decomposition of shape $\lambda/\mu$ into the border strips $(\rho_1,\dots,\rho_8)$ using the cutting strip $\rho$.}
    \label{fig:cutting edge}
\end{figure}

\begin{definition}[\cite{LP}]
A Lascoux--Pragacz decomposition of $\lambda/\mu$ is a tuple of non-intersecting lattice paths $(\theta_1,\dots, \theta_k)$, where the cutting strip $\rho=\theta_1$ is the outer borderstrip of $\lm$ (wrapping around $\lambda$). 

\end{definition}

\begin{definition}[height of $\theta_i$]\label{def:height}
Given a skew shape $\lambda/\mu$, let $(\theta_1,\dots,\theta_k)$ be its Lascoux--Pragacz decomposition. Denote by $\mathsf{f}(\theta_r)$ the row number of the final element of $\theta_r$ (top right-most element). We define the $\theta_r$-height of row $i$, denoted as $ht_{\theta_r}(i)$ to be $i - \mathsf{f}(\theta_r)$. 
\end{definition}
\begin{figure}
    \centering
    \includegraphics[height=4cm]{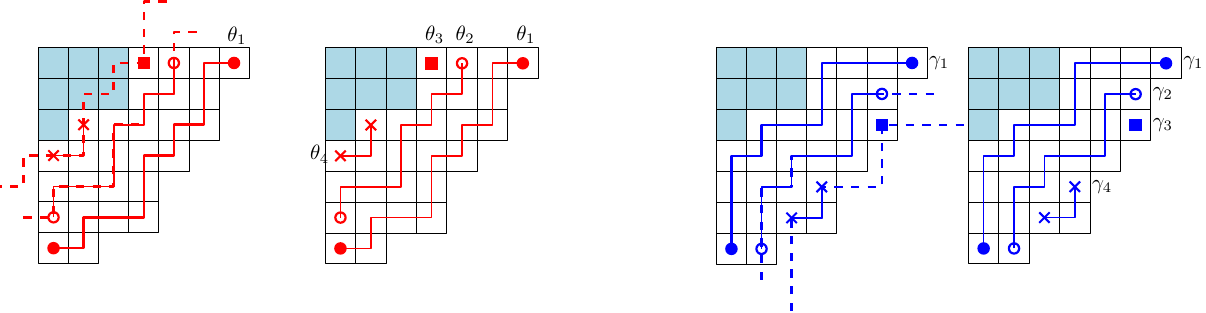}
    \caption{On the left we have  the diagonal slidings using the outer ($\theta_1$) border strip of $\lambda/\mu$ as the cutting strip and the Lascoux--Pragacz (red) decomposition of $\lambda/\mu$ and on the right we have the inner ($\gamma_1$) border-strip of $\lambda/\mu$ as the cutting edge and the Kreiman decomposition (blue) of  $\lambda/\mu$.}
    \label{exLascouxPragacz}
\end{figure}
\begin{example}
For the skew shape $7665442/331$ in  Figure~\ref{exLascouxPragacz}, the height of the element $(7,2)$ in $\theta_1$ is $6$ whereas the height of $(4,1)$ in $\theta_4$ is $1$.
\end{example}

\begin{definition}[\cite{Kre1}]
A \emph{Kreiman decomposition} of $\lambda/\mu$ is a tuple of non-intersecting lattice paths $(\gamma_1,\dots, \gamma_k)$, where the cutting strip $\rho=\gamma_1$ is the inner border strip of $\lambda/\mu$ (wrapping around $\mu$).
\end{definition}

Next, we give some results about the relationship between these two decompositions of $\lm$. To see how this works, let $(r_1,\ldots,r_m)$ be the lengths of the diagonals within $\lm$, enumerated from the upper right corner to the lower left. For example, in Figure~\ref{exLascouxPragacz} we have $(1,1,2,3,2,2,2,3,3,3,2,2,1)$. Now observe that cutting with an outer or inner border strip is removing a square from the lower or top of each diagonal, respectively, the remaining skew shape has diagonal lengths $(\max\{r_1-1,0\},\max\{r_2-1,0\},\ldots, \max\{r_m-1,0\})$. When there is an $r_i=0$ it signals the beginning of a new strip, these are numbered from top to bottom, or in the sequence $r$ setting from left to right.

\begin{proposition}\label{prop: cutting strip distance}
Given a skew shape $\lm$ and $(\theta_1,\dots, \theta_k)$ and $(\gamma_1,\dots, \gamma_k)$ be its Lascoux--Pragacz and Kreiman decompositions, respectively. Then for all $i = 1,\dots, k$, the strips $\gamma_i$ and $\theta_i$ have the same diagonal distance with respect to $\gamma_1$ and $\theta_1$, respectively.
\end{proposition}

\begin{proof}
We use the interpretation via diagonal lengths above to explain this phenomenon, as it makes it apparent. We proceed by induction on the distance $\epsilon$ from the original cutting strips. Let $(r_1,\ldots, r_m)$ be the diagonal lengths of $\lm$ as above. When $\epsilon=0$ we have $\gamma_1$ and $\theta_1$ respectively. The new diagonal lengths are $r(1):=(\max\{r_1-1,0\},\max\{r_2-1,0\},\ldots, \max\{r_m-1,0\})$. The strips at diagonal distance $\epsilon=1$ consist of the intervals between zeros in $r(1)$ and are numbered accordingly. In general, the strips at distance $\epsilon$ result from the intervals between zeros in $r(\epsilon):= (\max\{r_1-\epsilon,0\},\max\{r_2-\epsilon,0\},\ldots, \max\{r_m-\epsilon,0\})$ and are numbered consecutively from left to right. this procedure is independent of whether it follows the inner or outer border strip. Thus, the strips in $\gamma$ and $\theta$ at the same distance $\epsilon$ will have the same indices.

For example, in Figure~\ref{exLascouxPragacz} we have $r(0) = (1,1,2,3,2,2,2,3,3,3,2,2,1)$  giving both $\theta_1$ and $\gamma_1$, shifting one up/down we have $r(1) = (0,0,1,2,1,1,1,2,2,2,1,1,0)$ which results in one new strips (between the zeros) -- $\gamma_2$ and $\theta_2$ at distance 1 from the originals. At distance 2 we have $r(2)=(0,0,0,1,0,0,0,1,1,1,0,0,0)$ resulting in two strips -- $\theta_3,\theta_4$ and $\gamma_3,\gamma_4$ corresponding to the intervals between zero $(1)$ and $(1,1,1)$.

\end{proof}

We denote by $\epsilon_i$  the diagonal distance for both the border strips $\gamma_i$ and $\theta_i$ in the Kreiman and Lascoux--Pragacz decomposition of $\lm$.

Let $\NI(\lambda/\mu)$ be the set of  $k$-tuples of non-intersecting lattice paths contained in $[\lambda]$ where the $i$th path has endpoints $(a_i,b_i)$ and $(c_i,d_i)$.

We have the following correspondence between the Kreiman decomposition and the Lascoux--Pragacz decomposition of shape $\lambda/\mu$.

\begin{proposition}[{\cite[Lemma 3.8]{MPP1}}] \label{prop: corr gammas and thetas}
Given a shape $\lambda/\mu$, let  $(\theta_1,\dots, \theta_k)$ and $(\gamma_1,\dots, \gamma_k)$  be its Lascoux--Pragacz and Kreiman decompositions, respectively. For each $i \in 1,\dots, k$, the endpoints of $\gamma_i$ and $\theta_i$ have the same contents and therefore the two strips have the same length.
\end{proposition}

This correspondence between the border strips of the Lascoux–Pragacz and the Kreiman decomposition described in Proposition~\ref{prop: cutting strip distance} and \ref{prop: corr gammas and thetas} gives the following property.

\begin{corollary}
For each $i = 1\dots, k$, $\gamma_i$ and $\theta_i$ have the same length, number of columns, number of rows.
\end{corollary}

\begin{proof}
The fact that $\gamma_i$ and $\theta_i$ have the same length follows immediately from Proposition~\ref{prop: corr gammas and thetas} since the  endpoints of $\gamma_i$ and $\theta_i$ are on the same diagonal.

We show the other two equalities in the statement inductively. The statement is true for $\gamma_1$ and $\theta_1$ because they have the same endpoints and the number of horizontal and vertical steps are determined by the endpoints. We follow the construction of $\theta_i$ and $\gamma_i$ from the proof of Proposition~\ref{prop: corr gammas and thetas} in \cite[Lemma 3.8]{MPP2}. From the proof we know that for each $i = 2,\dots, k$, the shape $\lambda/\mu \setminus (\cup_{j=1}^{i-1}\gamma_j)$ is equal to the shape $\lambda/\mu \setminus (\cup_{j=1}^{i-1} \theta_j)$. Denote the shape $\eta^{(i)}/\mu$. Then the border-strips $\theta_i$ and $\gamma_i$ starts and ends at the same points of $\eta_i/\mu$ and the number of columns and rows of the two border-strips are equivalent.
\end{proof}

We will also need the following other technical relation between the two decompositions.

\begin{lemma}\label{lemma: gamma-theta}
Given a skew shape $\lambda/\mu$ and let $(\gamma_1,\dots,\gamma_k)$ and $(\theta_1,\dots,\theta_k)$ be its Kreiman and Lascoux--Pragacz decomposition. For each $i\in \{1,\dots,k\}$, let $\gamma_i:(f_i,g_i)\to(c_i,d_i)$ and $\theta_i:(s_i,t_i)\to (u_i,v_i)$. Then $g_i-t_i= \epsilon_i$, i.e. the column difference of the starting point of $\gamma_i$ and $\theta_i$ is $\epsilon_i$.
\end{lemma}

\begin{proof}
By construction the endpoints of $\theta_i$ are on the outer strip $\gamma_1$ and the endpoints of $\gamma_i$ are on the outer strip $\theta_1$. Since by Proposition~\ref{prop: corr gammas and thetas} the endpoints of $\theta_i$ and $\gamma_i$ are on the same diagonals, then the distance between $g_i-t_i$ between them is the same and equal to $\epsilon_i$.
\end{proof}

\subsection{Combinatorial objects of the Naruse hook length formula}

We go over some of the combinatorial objects that index the terms of \eqref{eq:nhlf} (see Figure~\ref{fig: objects Naruse}). We fix a shape $\lambda/\mu$.

\subsubsection*{Excited diagrams}
Excited diagrams were defined by Naruse and Ikeda \cite{IN} in the context of equivariant Schubert calculus and also independently by Kreiman \cite{Kre1} and Knutson--Miller--Yong \cite{KMY}. To introduce we need to define the notion of the following local move.

\begin{definition}[excited move $\beta$]
Given a subset $D$ of $[\lambda]$, a cell $(i,j)$ of $D$ is {\em active} if the cells $(i+1,j),(i,j+1),(i+1,j+1)$ are not in $D$ but are in $[\lambda]$. Given an active cell $(i,j)$ of $D$, let $\beta_{(i,j)}:D\to D'$ be the map that replaces cell $(i,j)$ in $D$ by $(i+1,j+1)$. We call such $\beta_{(i,j)}$ an \emph{excited move} (see Figure~\ref{fig:excited move}). 
\end{definition}

Then, we define excited diagrams of $\lm$ iteratively.

\begin{definition}
An \emph{excited diagram} of $\lambda/\mu$ to be any set of $|\mu|$ cells obtained by starting with the cells of $[\mu] \subseteq [\lambda]$ and applying any number of excited moves. We let $\mathcal{E}(\lambda/\mu)$ be the set of excited diagrams of $\lambda/\mu$. 
\end{definition}

Given an active cell $(i,j)$ in $D$, we denote by $\mu(i,j)$ the column of the original cell in $\mu$ where $(i,j)$ came from.

There is the following determinantal formula for the number of excited diagrams. Given a skew shape $\lm$, consider the diagonal that passes through the box $(i,\mu_i)$ and let $\vartheta_i$ be the row where the diagonal intersects the boundary of $[\lambda]$.

\begin{theorem}[{\cite[\S 3]{MPP1}}] \label{thm:det formula nED}
In the notation above, we have
\[
\nED(\lm) \,=\, \det \left[\binom{\vartheta_i + \mu_i -i+j-1}{\vartheta_i-1} \right]_{i,j=1}^{\ell(\mu)}.
\]
\end{theorem}

Next, we go over three objects in correspondence with excited diagrams. See Figure~\ref{fig:excited diagrams}.

\begin{figure}
    \centering
      \begin{subfigure}[normal]{0.4\textwidth}
\centering
\includegraphics{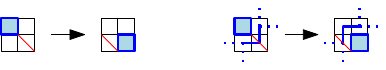}
    \caption{}
    \label{fig:excited move}
    \end{subfigure}
    \begin{subfigure}[normal]{0.5\textwidth}
    \centering   \includegraphics{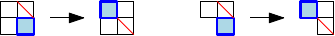}
    \caption{}
    \label{fig:reverse excited move}
    \end{subfigure}
    \caption{Illustration of (a) an excited move on excited diagrams and the corresponding ladder move on non-intersecting lattice paths, and (b) a reverse excited move. The broken diagonals are denoted in \textcolor{red}{red}.}
\end{figure}

\subsubsection*{Flagged tableaux of shape $\mu$}
Excited diagrams are in correspondence with certain SSYT of shape $\mu$.   Let $\mathcal{F}(\lambda/\mu)$ be the set of SSYT of shape $\mu$ with entries in row $i$ at most $\vartheta_i$. Such SSYT are called {\em flagged tableaux}. Given an excited diagram $D$ in $\ED(\lambda/\mu)$, let $\varphi(D):=T$ be the SSYT of shape $\mu$ where $T(x,y)=j_y$ where $(i_x,j_y)$ is the cell of $D$ corresponding to $(x,y)$ in $[\mu]$. 

\begin{proposition}[{\cite[\S3]{MPP1}}] \label{prop:phi bijection ED and flagged tableaux}
For a skew shape $\lm$, the map $\varphi$ is a bijection between $\ED(\lambda/\mu)$ and $\mathcal{F}(\lambda/\mu)$.
\end{proposition}

\subsubsection*{Non-intersecting lattice paths}
Kreiman showed  in \cite[\S5, \S6]{Kre1} that the support of such paths are the complements of excited diagrams. See Figure~\ref{fig:excited diagrams} for an example. Then excited moves on diagrams correspond to {\em ladder moves} on the non-intersecting lattice paths (see Figure~\ref{fig:excited diagrams}).

\begin{proposition}[\cite{Kre1}] 
The $k$-tuples of paths in $\NI(\lambda/\mu)$ are uniquely determined by their supports in $[\lambda]$ and moreover these supports are exactly the complements of excited diagrams of $\lambda/\mu$.
\end{proposition}

If an excited diagram $D$ corresponds to the tuple $(\gamma'_1,\ldots,\gamma'_k)$, we denote by $\gamma_i(D)=\gamma'_i$, the path obtained from $\gamma_i$ after applying the corresponding ladder moves to the excited moves where $[\mu]$ is obtained from $D$.

\subsubsection*{Broken diagonals of excited diagrams}
To each excited diagram $D\in
\ED(\lambda/\mu)$ we associate certain cells of $[\lm]\setminus D$. See Figure~\ref{fig:excited diagrams} for an example.

\begin{definition}[{\cite[Def. 7.3]{MPP1}}] \label{def:broken diagonals}
Given an excited diagram $D$ in $\ED(\lm)$, we define the {\em broken diagonals} of $D$ to be the following set of $[\lm]\setminus D$:
\begin{itemize}
\item for the initial diagram $[\lm]$, let $Br(D)= \bigcup_{i=1}^{\ell(\lambda)-1} \dd_i(\mu)$, where $\mathsf{d}_k$ is the diagonal of cells in $[\lm]$ with content $\mu_i-i$. 
\item if $D'=\beta_{(i,j)}(D)$, then $(i+1,j+1)$
is in some $\dd_t(D)$. Define $\Br(D')=\bigcup_{k=1}^{\ell(\lambda)-1} \dd_k(D')$ where $\dd_k(D')=\dd_k(D)$ if $k\neq
t$ and $\dd_k(D')=\{(i+1,j)\} \cup \dd_k(D) \setminus \{(i+1,j+1)\}
$. 
\end{itemize}
Let $\Br(\lambda/\mu) = \{\Br(D) \mid D \in \ED(\lm)\}$ be the set of broken diagonals of the excited diagrams of $\lm$. 
\end{definition}

For a lattice path $\gamma\subseteq [\lambda]$ let $u(\gamma)$ be the cells corresponding to up-steps of $\gamma$, including right-up corners (i.e. outside corners) but excluding up-right corners (i.e. inner corners). The following result gives a non-recursive characterization of broken diagonals.

\begin{proposition} \label{prop:broken diagonals in gamma}
Let $D$ be an excited diagram $D$ in $\ED(\lm)$ corresponding to the tuple $(\gamma'_1,\ldots,\gamma'_k)$ in $\NIP(\lm)$ then the broken diagonals are placed on both the vertical steps or a right-up corners (i.e. outside corners) of the paths $\gamma'_i$:
\[
\Br(D) = \{u(\gamma_i) \mid i=1,\ldots,k\}.
\]
\end{proposition}

\begin{proof}
The size of the up-steps $\{u(\gamma_i) \mid i=1,\ldots,k\}$ remains constant for all the tuples of paths in $\NIP(\lm)$. The result holds for the initial excited diagram $D=[\mu]$ and an excited move changes an up-step to an up-step in an adjacent column.
\end{proof}

\begin{figure}
 \begin{subfigure}[normal]{0.4\textwidth}
     \centering
      \includegraphics[scale=0.8]{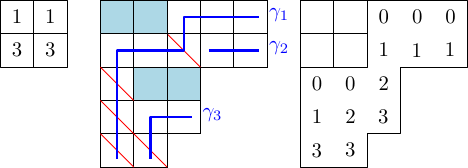}
    \caption{}
    \label{fig: objects Naruse}
    \end{subfigure}    
    \qquad 
     \begin{subfigure}[normal]{0.5\textwidth}
     \centering
      \includegraphics[scale=0.8]{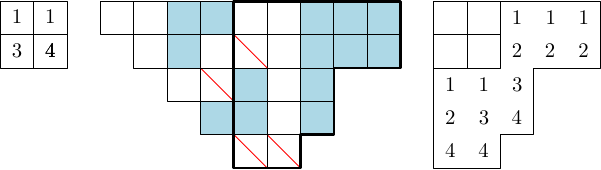}
    \caption{}
    \label{fig: objects OOF}
    \end{subfigure}
    \caption{Example of corresponding objects of a (a) term of \eqref{eq:nhlf}: flagged tableau, excited diagram (with \textcolor{red}{broken diagonals} and non-intersecting \textcolor{blue}{paths}), and minimal SSYT shape $\lm$; and of a (b) term of the \eqref{eq:OO}: Okounkov--Olshanski tableau, reverse excited diagram, and flagged tableau shape $\lm$.}
    \label{fig: objects NHLF and OOF}
\end{figure}

\subsection{Combinatorial objects of the  Okounkov--Olshanski formula}

We go over some of the combinatorial objects that index the terms of \eqref{eq:OO} (see Figure~\ref{fig: objects OOF}). We fix a shape $\lambda/\mu$. 

\subsubsection*{Okounkov--Olshanski tableaux} 
For a skew shape $\lm$ of length $d$, an {\em Okounkov--Olshanski tableau} is a SSYT $T$ of shape $\mu$ with entries in $[d]$ where $d:=\ell(\lambda)$ and with $c(u) < \lambda_{d + 1 - T(u)}$ for all $u \in [\mu]$. Let $\cOOT(\lambda/\mu)$ be the set of such tableaux and let $\nOOT(\lambda/\mu)$ be the size of $\cOOT(\lm)$.

\subsubsection*{Skew flagged tableaux} 
Let $\mathcal{SF}(\lambda/\mu)$ be the set of SSYT of shape $\lambda/\mu$ such that the entries in row $i$ are at most $i$. The authors in \cite{MZ} gave an  explicit bijection between tableaux in $\cOOT(\lambda/\mu)$ and skew tableaux in $\mathcal{SF}(\lambda/\mu)$.

Given $T\in \mathcal{SF}(\lm)$, $M(T)$ is the set of cells $(j,t)$ with $1\leq j\leq \ell(\lambda')$ and $1\leq t \leq \lambda'_j$ such that there is no $i$ with $T(i,j)=t$. Given $(j,t) \in M$, let $i(j,t)$ be the unique integer $i$ with $\mu'_i \leq i \leq \lambda'_j$ so that $T(i',j)<t$ for $i'\leq i$ and $T(i',j)>t$ for $i'>i$. i.e. view $M(T)$ as the set of labels on the horizontal edges of $T$ between the cells  $(i(j,t),j)$ and $((i(j,t)+1,j)$ so that column $j$ of $T$ has all the numbers from $1$ to $\lambda'_j$ from top to bottom.

 \subsubsection*{Reverse excited diagrams}
The following variation of excited diagrams was introduced in \cite{MZ}. 

\begin{definition}[reverse excited move $\beta$]
Given a subset $D$ of $[\lambda^*]$, a cell $(i,j)$ with $i\neq j$  of $D$ is {\em reverse active} if the cells $(i-1,j),(i,j-1),(i-1,j-1)$ are not in $D$ but are in $[\lambda^*]$. Similarly, a cell $(i,i)$ of $D$ is {\em reverse active} if the cells $(i,j-1),(i-1,j-1)$ are not in $D$ but are in $[\lambda^*]$. Given a reverse active cell $(i,j)$ of $D$, let $\beta^*_{(i,j)}:D\to D'$ be the map that replaces cell $(i,j)$ in $D$ by $(i-1,j-1)$. We call such $\beta^*_{(i,j)}$ a \emph{reverse excited move} (see Figure~\ref{fig:reverse excited move}). 
\end{definition}

Then, we define reverse excited diagrams of $\lm$ iteratively.

\begin{definition}
A \emph{reverse excited diagram} of $\lambda/\mu$ to be any set of $|\lambda/\mu|$ cells obtained by starting with the cells of $[\lambda^*/\mu^*] \subseteq [\lambda^*]$ and applying any number of reverse excited moves. We let $\mathcal{RE}(\lambda/\mu)$ be the set of reverse excited diagrams of $\lambda/\mu$. 
\end{definition}

The flagged tableaux in $\mathcal{SF}(\lambda/\mu)$  are in direct correspondence with reverse excited diagrams in $\mathcal{RE}(\lm)$. Given a reverse excited diagram $D$ in $\mathcal{RE}(\lm)$, let $\varphi^*(D):=T$ be the SSYT of shape $\lm$ where $T(x,y)=j_y$ where $(i_x,y_x)$ is the cell of $D$ corresponding to $(x,y)$ in $[\lambda^*/\mu^*]$. 

\begin{proposition}[{\cite[Lem. 4.22]{MZ}}]
For a skew shape $\lm$, the map $\varphi^*$ is a bijection between $\mathcal{RE}(\lm)$ and $\mathcal{SF}(\lm)$.    
\end{proposition}

The direct correspondence between tableaux in $\mathcal{OOT}(\lm)$ and in $\mathcal{SF}(\lm)$ is as follows. Given $T$ in $\mathcal{OOT}(\lm)$, let $\vartheta^*(T):=U$ be the SSYT of shape $\lm$ whose entries in the $i$th column are the entries $k$ in $\{1,\ldots,\lambda'_i\}$ such that $d+1-k$ is not an entry in the $i$th column of $T$. 

\begin{proposition}[{\cite[Cor. 4.16]{MZ}}] \label{bij:OOT and flagged skew tableaux}
For a skew shape $\lm$, the map $\vartheta^*$ is a bijection between $\mathcal{OOT}(\lm)$ and $\mathcal{SF}(\lm)$.    
\end{proposition}
 
There is also the following determinantal formula for the number of Okounkov--Olshanski tableaux. 

\begin{theorem}[{\cite[Thm. 1.2]{MZ}}] \label{thm:det formula nOOF}
In the notation above, we have
\[
\nOOT(\lm) \,=\, \det \left[\binom{\lambda_i-\mu_j+j-1}{i-1} \right]_{i,j=1}^{\ell(\lambda)}.
\]
\end{theorem}

 \subsubsection*{Broken diagonals of reverse excited diagrams}
 
To each reverse excited diagram $D\in
\mathcal{RE}(\lambda/\mu)$ we associate certain cells of $[\mu^*]\setminus D$. See Figure~\ref{fig:oof_compare} for an example.

\begin{definition}[{\cite[Def. 4.24]{MZ}}] \label{def:broken diagonals in RE}
Given a reverse excited diagram $D$ in $\mathcal{RE}(\lm)$, we define the {\em broken diagonals} of $D$ to be the following set of $[\lambda^*]\setminus D$:
\begin{itemize}
\item for the initial diagram $[\lambda^*/\mu^*]$, let $\Br^*(D)= \bigcup_{i=1}^{\ell(\mu')} \dd^*_i(\mu)$, where $\dd^*_k$ is the diagonal of cells in $[\mu^*]$ with content $i-\mu'_i+\ell(\lambda)-1$.
\item if $D'=\beta^*_{(i,j)}(D)$, then $(i-1,j)$ is in some $\dd^*_t(D)$. Define $\Br^*(D')=\bigcup_{k=1}^{\ell(\mu')} \dd^*_k(D')$ where $\dd^*_k(D')=\dd^*_k(D)$ if $k\neq
t$ and $\dd^*_k(D')=\{(i,j)\} \cup \dd^*_k(D) \setminus \{(i-1,j)\}
$. 
\end{itemize}
\end{definition}

The reformulation of \eqref{eq:OO} in terms of reverse excited diagrams is simple to state \cite[Cor. 5.6]{MZ}.

\begin{corollary}[Reverse excited diagram formulation of \eqref{eq:OO}] \label{cor:rev excited-form}
For a skew shape $\lambda/\mu$ of size $n$ and length $d=\ell(\lambda)$ we have that
\begin{equation} \label{eq:cor:flag-form}
f^{\lm} = \frac{n!}{\prod_{u \in [\lambda]} h(u)} \sum_{D \in \mathcal{RE}(\lambda/\mu)} \prod_{(i,j) \in \Br^*(D)} (\lambda_i + d-j).
\end{equation}
\end{corollary}

\subsection{The Hillman--Grassl correspondence} \label{sec: HG background}

\subsubsection*{Definition and properties}

Fix a skew shape $\lambda/\mu$. We denote by $\HG$ the {\em Hillman--Grassl bijection} between reverse plane partitions in $\RPP(\lambda)$ ranked by size and integer arrays of shape $\lambda$ ranked by hook weight. That is, if $\HG:\pi \mapsto A$ then 
\begin{equation} \label{eq: HG hook weight}
|\pi|=\sum_{u \in \lambda} h(u)\cdot A_u.
\end{equation}

We follow the conventions and definition of $\HG$ in \cite[\S 7.22]{EC2}, that we include for completeness.  

\begin{definition}[Hillman--Grassl map]\label{def: HG}
Given $\pi \in \RPP(\lambda)$, the map is obtained via a  sequence of pairs $(\pi_0,A_0):=(\pi,{\bf 0}),(\pi_1,A_1)\ldots,(\pi_k,A_k):=({\bf 0},\HG(\pi))$ where  each $\pi_i$ is a reverse plane partition and $A_i$ is an array of nonnegative integers of shape $\lambda$. 

The HG step $(\pi_i,A_i) \to (\pi_{i+1},A_{i+1})$ is defined as follows. The plane partition $\pi_{i+1}$ is obtained from $\pi_i$ by decreasing by one the following entries along a lattice path on $[\lambda]$ defined as
\begin{itemize}
    \item[(i)] start at the most South-West  nonzero cell of $\pi_i$, 
    \item[(ii)] iterate the following step when the path visits cell $(a,b)$, if $(\pi_i)_{a,b}=(\pi_i)_{a-1,b}$ then the path moves North, otherwise the path moves East,
    \item[(iii)] terminate when this is no longer possible. 
\end{itemize}
Let $u_i=(r,c)$ where $c$ is the column where the path starts and $r$ is the row where it ends.
Note that $|\pi_i|-|\pi_{i-1}|=h(u_i)$. We obtain the array $A_{i+1}$ from $A_i$ by adding one to cell $u_i$. 
We continue until $\pi_k$ has only zero entries. We let $A_k=\HG(\pi)$.  
\end{definition}

Note that by the construction of $\HG$, \eqref{eq: HG hook weight} holds. Hillman--Grassl showed that this map is a bijection. Because of this and \eqref{eq: gf SSYT RPP lambda}, we obtain \eqref{eq:qhlf} as a corollary.

\begin{theorem}[{Hillman--Grassl \cite{HG}}]
The map $\HG$ is a bijection between $\RPP(\lambda)$ and integer arrays of shape $\lambda$ satisfying \eqref{eq: HG hook weight}.
\end{theorem}

The Hillman--Grassl correspondence has the following property that allows to encode different traces of the plane partition. For an integer $k$ with $1-\ell(\lambda)\leq k \leq \lambda_1-1$, a $k$-diagonal of $\pi$ is the sequence of entries $(\pi_{ij})$ with content $i-j=k$. The {\em $k$-trace} of $\pi$ is the sum $\tr_k(\pi)=\sum_{i-j=k} \pi_{ij}$.

For an integer $k$ with $1-\ell(\lambda)\leq k \leq \lambda_1-1$, let $\square_k^{\lambda}$ be the largest $i\times (i+k)$ rectangle that fits inside $[\lambda]$ starting at $(1,1)$. For $k=0$, we have that $\square_0^{\lambda}$ is the {\em Durfee square} of $\lambda$. If $A$ is an array of shape $\lambda$, let $A_k$ be the subarray of $A$ with support in $\square_k^{\lambda}$, and $|A_k|$ denotes the sum of the entries of this subarray.

 \begin{proposition}[{\cite[Thm. 3.2]{G}}] \label{prop:trace and HG}
 Let $\pi$ be a reverse plane partition of shape $\lambda$, and $A=\HG(\pi)$. Then for $k$ with $1-\ell(\lambda)\leq k \leq \lambda_1-1$, we have that $\tr_k(\pi)=|A_k|$.
 \end{proposition}

 Denote by $ac_t(A_k)=\max \{ \sum_{(i,j) \in p_1\cup \cdots \cup p_t} A_{i,j}\} $, where $p_1,\ldots,p_t$ are collections of northeast (NE) non-intersecting paths in $A_k$. Similarly, let $dc_t(A_k)$ denote the maximal number of nonzero entries along a collection of strict southeast (SE) non-intersecting paths in $A_k$. The statistic $ac_t(A)$ ($dc_t(A)$) can be viewed as counting the maximum combined length of $t$ {\em ascending (descending) chains} in $A$. We need the following known connection between the Hillman--Grassl bijection and the RSK bijection that can be viewed as an analogue of {\em Greene's theorem} (see \cite[Thm. A1.1.1]{EC2}).

\begin{theorem}[{(i) by Hillman--Grassl \cite{HG}, (ii) by Gansner \cite{G}}] \label{thm: Greene's theorem for HG}
Let $\pi$ be in $\RPP(\lambda)$, $A=HG(\pi)$, and let $k$ be an integer $1-\ell(\lambda) \leq k \leq \lambda_1-1$. Denote by $\nu=\nu^{(k)}$ the partition whose parts are the entries on the $k$-diagonal of $\pi$. Then for al $t\geq 1$ we have:
\begin{itemize}
    \item[(i)] $ac_t(A_k) = \nu_1+\nu_2+\cdots + \nu_t$,
    \item[(ii)] $dc_t(A_k) = \nu'_1+\nu'_2+\cdots + \nu'_t$.
\end{itemize}
\end{theorem}

\begin{corollary}[{\cite[Cor. 5.8]{MPP1}}] \label{cor: rel HG and RSK}
Let $\pi$ be in $\RPP(\lambda)$, $A=HG(\pi)$, and let $k$ be an integer $1-\ell(\lambda) \leq k \leq \lambda_1-1$. Denote by $\nu^{(k)}$ the partition obtained from the $k$-diagonal of $\pi$. Then the shape of the tableaux of $RSK(A_k^{\updownarrow})$ and of $RSK(A_k^{\leftrightarrow})$ is $\nu^{(k)}$.
\end{corollary}

\subsubsection*{Extending $\HG$ to skew SSYT}

We extend $\HG$ to $\SSYT(\lambda/\mu)$ by viewing each skew SSYT as a plane partition of shape $\lambda$ with zero entries in $[\mu]$. 

For any excited diagram $D\in \E(\lambda/\mu)$, denote by $A_D$ the $0$-$1$ array of shape $\lambda$ with support on the broken diagonals $\Br(D)$ of $D$ and let $\mathcal{A}^*(\lm)=\{A_D \mid D\in \ED(\lm)\}$. Let $\mathcal{A}^*_D$ be the set of arrays $A$ of nonnegative integers of shape $\lambda$ with support contained in $[\lambda]\setminus D$ and nonzero entries $A_{i,j}>0$ if $(A_D)_{i,j}=1$. 

\begin{theorem}[{\cite[Thm. 7.7]{MPP1}}] \label{thm: HG skew SSYT}
The (restricted) Hillman--Grassl map $\HG$ is a bijection
\[
\HG: \SSYT(\lm) \to \bigcup_{D\in \ED(\lm)} \mathcal{A}^*_D.
\]
\end{theorem}

This property combined with \eqref{eq: gf SSYT RPP lambda} yields \eqref{eq:qnhlf} as a corollary (see \cite[Sec. 7.1]{MPP1}).

\section{Minimimal semistandard tableaux of shape \texorpdfstring{$\lambda/\mu$}{lambda/mu}}\label{sec: minimal SSYT}

In Section~\ref{sec:def min ssyt} we introduce a set of tableaux called \emph{minimal SSYT} of shape $\lambda/\mu$. In Section~\ref{sec: bijection min ssyt to ED} we show it is in correspondence with excited diagrams.

\subsection{Characterization of minimal SSYT by local moves} \label{sec:def min ssyt}

We now give the definition of minimal SSYT of shape $\lm$ and show they are in correspondence to the set of excited diagrams of shape $\lambda/\mu$. We do this by studying the Lascoux--Pragacz decomposition of the SSYT and the Kreiman decomposition of its corresponding excited diagram.

Let $\lambda/\mu$ be a skew partition and $(\theta_1,\dots,\theta_k)$ be its Lascoux--Pragacz decomposition. Given a strip $\theta_k$ and column $j$, let $\theta_k(j)$ the $j$th column segment of $\theta_k$. Let  $(i,j)$ be the top-most cell  and $(i',j)$ be the bottom-most cell of $\theta_k(j)$ (which may agree if the column has size one). 

Recall that $T_0 \in \SSYT(\lambda/\mu)$ is the \emph{minimum SSYT of shape} $\lambda/\mu$, the tableau whose $i$-th column is $\{0,1,\dots,\lambda_i' - \mu_i' -1\}$. A \emph{minimal SSYT} is obtained from $T_0$ by applying a sequence of \emph{excited moves} $\delta$, defined below. 
This move increases by $1$ the entries in a column segment of a strip $\theta_k$ in a tableau when the result is still semistandard and the top value of the segment is at most its distance from the top part of $\theta_k$.

\begin{definition}[excited move $\delta$] \label{def:delta move}
The top most cell $(i,j)$ of a column segment $\theta_k(j)$ of a SSYT $T$ of shape $\lm$ is  called \emph{active} if 
\begin{enumerate}
    \item[(i)] $T(i,j) < ht_{\theta_k}(i)$, 
    \item[(ii)] $T(i,j)< T(i, j+1)$ and $T(i',j)< T(i'+1,j)-1$, where $(i',j)$ is the cell at the bottom of $\theta_k(j)$
\end{enumerate}
Given an active column $\theta_k(j)$ of $T$, the \emph{excited move} $\delta$ adds one to each entry in the column segment $\theta_k(j)$. We denote by $T':=\delta_{(k;j)}(T)$ be the tableau obtained from $T$ by the move $\delta$. Condition~(ii) guarantees that $T'$ is still a SSYT. See Figure~\ref{fig: ssyt min move} for an example. 
\end{definition}

\begin{definition}[minimal skew SSYT]
A {\em minimal SSYT} of shape $\lm$ is any tableau obtained by starting with $T_0$ and applying any number of excited moves $\delta$. We let $\SSYT_{\min}(\lm)$ be the set of minimal SSYT of shape $\lm$.  See Figure~\ref{fig: min SSYT example} for an example.
\end{definition}

\subsection{Bijection between excited diagrams and minimal SSYT} \label{sec: bijection min ssyt to ED}

In this section, we show the correspondence between $\E(\lambda/\mu)$ and $\SSYT_{\min}(\lambda/\mu)$ that commutes with their respective excited moves $\beta$ and $\delta$. First, we define the map $\Phi$ between $\E(\lambda/\mu)$ and $\SSYT(\lambda/\mu)$ and show it is a bijection. We then show the correspondence between the excited moves in $\E(\lambda/\mu)$ and the excited moves in $\SSYT_{\min}(\lambda/\mu)$.

The map $\Phi$ is defined below. Intuitively, it builds a tableau of shape $\lm$ from an excited diagram by filling the entries in the strip $\theta_i$ by counting the number of broken diagonals column by column of $\gamma_i(D)$.

\begin{definition}\label{def: phi* bijection}
Given $D\in \E(\lambda/\mu)$ with Kreiman decomposition $(\gamma_1(D),\ldots,\gamma_k(D))$ and broken diagonals $\Br(D)$, let $T:=\Phi(D)$ be the following tableau of shape $\lambda/\mu$: 

For each $\gamma_i(D)$ and $\theta_i$, start with the left-most column of each path. For each $j$th column of $\gamma_i(D)$ and $\theta_i$, we do the following procedure:

\begin{enumerate}
    \item Let $b_j$ be the number of broken diagonals on the $j$th column of $\gamma_i(D)$.
    \item Denote the bottom-most element of the $j$th column of $\theta_i$ as $(s_j,t_j)$. Let 
    \[
    T(s_j,t_j) = 
    \begin{cases}
    b_j & \text{ if } j=1,\\
    T(s_j,t_j-1) + b_j & \text{ if } j>1.
    \end{cases}
    \]
    \item Fill the rest of the $j$th column of $\theta_i$ so that consecutive entries differ by $1$.
\end{enumerate}
\end{definition}

\begin{example}
Let $D\in \E(\lambda/\mu)$ be the excited diagram shown in Figure~\ref{fig: ssyt min} (b). We apply the correspondence to obtain $T = \Phi(D)$. The first column of $\gamma_1(D)$ has $3$ broken diagonals, so $T(5,1) = 3$. the second column of $\gamma_1(D)$ has no broken diagonals, so $T(5,2)=3$, then the rest of the column of $\theta_1$ is filled with consecutive entries differing by $1$ to maintain column strictness. We continue the algorithm to obtain the final tableau $T$ of skew shape $\lm$.
\begin{figure}[t]
     \begin{subfigure}[normal]{0.4\textwidth}
     \centering
    \includegraphics[height=2.5cm]{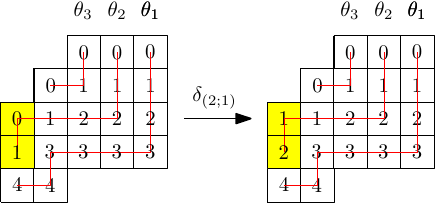}
    \caption{}
    \label{fig: ssyt min move}
    \end{subfigure}
    \quad
    \begin{subfigure}[normal]{0.5\textwidth}
    \centering
    \includegraphics[height=2.5cm]{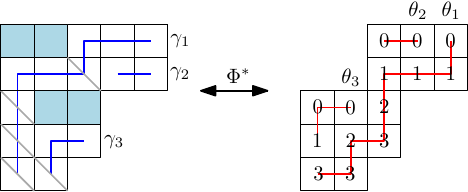}
    \caption{}
    \end{subfigure}
    \caption{(a) Excited move $\delta$ on $\SSYT_{\min}$ (b) Excited diagram to $\SSYT_{\min}$ }
    \label{fig: ssyt min}
\end{figure}
\end{example}

\begin{remark}
See Section~\ref{sec: HG bijection on min SSYT} for a description of $\Phi$ in terms of the Hillman--Grassl correspondence.
\end{remark}

The main result of this section is the fact that $\Phi$ is a bijection between $\ED(\lm)$ and $\SSYT_{\min}(\lm)$.

\bijectionEDandSSYTmin

The proof requires the following lemmas and notation. For the Lascoux--Pragacz decomposition of $\lm$ and the Kreiman paths $(\gamma_1,\ldots,\gamma_k)$ of an excited diagram $D$, we denote by $\col(\theta_i,n)$ and $\col(\gamma_i,n)$ the column on $\lm$ where the $n$th column of $\theta_i$ and $\gamma_i$ is, respectively. Note that $\lm$ will be assumed.

\begin{lemma} \label{lemma: base case}
We have that $\Phi([\mu])=T_0$.
\end{lemma}

\begin{proof}
 Let $T=\Phi(\mu)$. It suffices to show that $T(i,j)=i-\mu'_j-1$. We show this by iterating on the number $k$ of paths in the Kreiman and Lascoux--Pragacz decomposition. Consider the first paths $\gamma_1$ and $\theta_1$ of the shape $\lambda/\mu$. By Proposition~\ref{prop:broken diagonals in gamma}, the broken diagonals in the excited diagram $[\mu]$ are on the vertical steps and right up corners of $\gamma_i$.  Note that the Kreiman path $\gamma_1$ of the shape $\lambda/\mu$ traces around $[\mu]$. 

The bottom-most value in the $j$th column of $T(\theta_1)$ is $\lambda'_i-\mu'_i-1$. Indeed, by the description of $\Phi$ in Definition~\ref{def: phi* bijection}, the number $b_j$ of broken diagonals on the $j$th column of $\gamma_1$ is $b_1=\lambda'_1-\mu'_1-1$ and $b_j=\mu'_{j-1}-\mu'_j$, for $j>1$. The values on the bottom-most elements determine the rest of $T(\theta_1)$, since the values of the $j$th column of $\theta_1$ differ by $1$. Note that for $(i,j)$ in $\theta_1$, we have that $T(i,j)=i-\mu'_j-1$.

Next, consider the shape $\nu:= (\lambda/\mu) \setminus \theta_1$. Note that $\nu$ is equivalent to $(\lambda/\mu) \setminus \gamma_1$ and sliding up the rest of the Kreiman paths diagonally (see \cite[Lemma 3.8]{MPP2}). Then similarly, the bottom-most value in the $i$th column of $T(\theta_2)$ is $\nu'_i - \mu'_i -1$, which determine the rest of $T(\theta_2)$. Continuing this pattern, we obtain that $T(i,j)=i-\mu'_j-1$ as desired.
\end{proof}

\begin{lemma}\label{lemma: well-defined and inj}
The map $\Phi$ is a well-defined injective map from $\E(\lambda/\mu)$ to $\SSYT(\lambda/\mu)$.
\end{lemma}

\begin{proof}

First, we show that for any $D\in \E(\lambda/\mu)$, the tableau $\Phi(D)$ of shape $\lm$ is semistandard. We use induction on the number of excited moves.

The base case follows by Lemma~\ref{lemma: base case} and the fact that $T_0$ is semistandard. Next, assume that $T:=\Phi(D)$ is semistandard, and $D'=\beta_u(D)$ for some active cell $u=(x,y)$.  We show that $T':=\Phi(D')$ is semistandard.

Let $(\gamma_1,\ldots,\gamma_k)$ be the Kreiman decomposition of $D$, and $\gamma_i$ be the path involved in the corresponding ladder move to $\beta_u$. By Definition~\ref{def: phi* bijection}, $T'$ is semistandard within the path $\theta_i$ and some of its entries were increased. If $i=1$, then we are done since $\theta_1$ is the outer rim of $\lambda$, its entries increased and satisfy the SSYT inequalities with respect to its neighbors. If $i>1$, we show that $T'$ is semistandard between adjacent paths $\theta_i$ and $\theta_j$ for $i>j$, for $i<j$ this follows analogously as for $i=1$.   By construction of the Kreiman paths and Lascoux--Pragacz paths, since $i>j$ then  $\gamma_{i}$ is to the right and below of $\gamma_{j}$ and  $\theta_{i}$ is to the left and above of $\theta_{j}$. Since $\theta_j$ is adjacent to $\theta_i$, then $\epsilon_i-\epsilon_j =1$, and by Proposition~\ref{prop: cutting strip distance}, $\gamma_j$ is also adjacent to $\gamma_i$.

Suppose $\beta_u$ acted on the $n$th column of $\gamma_i$, that is, $\beta_u(D)$ increased by $1$ the size of broken diagonals on the $n$th column of $\gamma_i$. By the construction of $\Phi$, this corresponds to the $n$th column segment of $\theta_i$ increasing by $1$. Let $(a,b)$ be the bottom most cell in the $n$th column segment of $\theta_i$ and $(a+1,b)$ be the cell below in the $m$th column of $\theta_j$. We need to show that $T'(a+1,b) > T'(a,b)$.

In order to do this, we need to track the cell $(a+1,b)$ in the excited diagram. We know that if $(a,b)$ is in the $n$th column of $\theta_i$, then $(a+1,b)$ is in the $m$th column of $\theta_{j}$ where $m\geq n$ since $\theta_{i}$ is to the left and above of $\theta_j$. By the construction of $\Phi$, the $m$th column of $\theta_j$ corresponds to the $m$th column of $\gamma_j$.

We show that the $m$th column of $\gamma_j$ is directly to the left of the $n$th column of $\gamma_i$. Let $g_i$ and $t_i$ be the starting column of $\gamma_i$ and $\theta_i$ respectively. Then $\col(\gamma_i,n) = g_i+n -1$ and $\col(\theta_i,n) = t_i+n-1$. By Lemma~\ref{lemma: gamma-theta} we have that  $g_i-t_i = \epsilon_i$, and so $\col(\theta_i,n) = \col(\gamma_i,n) - \epsilon_i$. Similarly $\col(\theta_j,m) = \col(\gamma_j,m) - \epsilon_j$. Note that the $m$th column of $\theta_j$ is on the same column of $\lm$ as is the $n$th column of $\theta_i$, i.e. $\col(\theta_i,n) = \col(\theta_j,m)$. Then $\col(\gamma_i,n) - \col(\gamma_j,m) = \epsilon_i - \epsilon_j = 1$. This shows that the $m$th column of $\gamma_j$ is directly to the left of $n$th column of $\gamma_i$ with respect to $\lm$.

We now need to analyze the previous excited moves on the $m$th column of $\gamma_j$. Let $u$ be the active cell of the excited move $\beta_u(D)$. Given a sequence of excited move $\beta^{(1)},\cdots ,\beta^{(\ell)}$ such that $\beta^{(\ell)}\circ \cdots \circ \beta^{(1)} ([\mu]) = D$, let $\beta^{(r)}=\beta_{(x-1,y-1)}$ be the excited move that moved $(x-1,y-1)$ to $u=(x,y)$. Such an index $r$ exists, since $u=(x,y)$ in $D$ is not in $[\mu]$, i.e. $u\neq \mu(u)$. Let $D'':= \beta^{(r-1)}\circ \cdots \circ \beta^{(1)}([\mu])$ and $T'' = \Phi(D'')$. Note that by the choice of $r$, the values of cells $(a,b)$ and $(a+1,b)$ remain the same from $\Phi(\beta^{(r)}(D^{''}))$ to $T$. So it is enough to consider the values of these cells in $T^{''}$, $T$, and $T'$, respectively.  By the induction hypothesis, $T''$ and $T$ are semistandard, so 
\[
T''(a+1,b)> T''(a,b), \qquad T(a+1,b)> T(a,b).
\]
Since the excited move $\beta_x$ on $D''$ increased by $1$ the number of the broken diagonals on the $m$th column of $\gamma_j$ and did not change the number of the broken diagonals on the $n$th column of $\gamma_j$. Thus by  definition of $\Phi$, we have that 
\[
T(a+1,b) = T''(a+1,b)+1, \qquad T(a,b) = T''(a,b).
\]
This implies that $T(a+1,b)>T(a,b)+1$. Similarly, the same analysis on the excited move $\beta_u$ on $D$ implies that 
\[
T'(a,b) = T(a,b) + 1, \qquad T'(a+1,b) = T(a+1,b).
\]
Therefore we have $T'(a+1,b) > T'(a,b)$, as desired. This shows that $\Phi$ is well-defined.

Next, we show that $\Phi$ is injective. Note that for $D \neq D'$ in $\E(\lambda/\mu)$, their corresponding Kreiman paths $(
\gamma_1,\ldots,\gamma_k)$ and $(\gamma'_1,\ldots,\gamma'_k)$ in $\NIP(\lm)$ differ.  Thus there is an index $i$ where $\gamma_i\neq \gamma'_i$. Since the endpoints of these paths are the same there exists a column where these paths have a different number of broken diagonals. Thus the tableaux $\Phi(D)$ and $\Phi(D')$ are different too.
\end{proof}

\begin{lemma}[Commutation $\beta$ and $\delta$]\label{lemma: beta and delta}
Given an excited diagram $D$ in $\mathcal{E}(\lambda/\mu)$ and an active cell $u=(a,b)$ of $D$ we have that
\[
\Phi (\beta_u(D)) = \delta_{f(u)}(\Phi(D)),
\]
where $f(a,b)=(i;\mu(a,b))$, $i$ is the index of the Kreiman path $\gamma_i(D)$ modified by $\beta(a,b)$, and $\mu(a,b)$ is the column of the original cell in $\mu$ where $(a,b)$ came from.
\end{lemma}

\begin{proof}
The excited move $\beta_{(a,b)}$ corresponds to a ladder move on $\gamma_i(D)\to \gamma_i(D')$, for some $i$ and a broken diagonal shift from $(a+1, b+1)$ to $(a+1,b)$. Let $g_i$ and $t_i$ be the starting columns of $\gamma_i(D)$ and $\theta_i$ respectively. Suppose $(a+1,b)$ is in the $j$th column of the path $\gamma_i(D)$. Then the excited move $\beta_{(a,b)}$ shifts a broken diagonal to the $j$th column of $\gamma_i(D')$, where $j = b-g_i+1$. The $j$th column of $\gamma_i(D)$ corresponds to the $j$th column of $\theta_i$, which in turn corresponds  to  column $t_i+j-1$ of shape $\lambda/\mu$. By Proposition~\ref{lemma: gamma-theta} we know that $g_i-t_i = \epsilon_i$. Thus we have that 
\[
\col_{\lm}(\theta_i, j) = t_i + j -1 = (g_i-\epsilon_i) + (b-g_i + 1) -1 = b- \epsilon_i.
\]
Next, number of times the cell $(a,b)$ has been excited is $b-\mu(a,b)$. This quantity is also equal to  the number of paths $\gamma_j(D)$ that cross the diagonal of $(a,b)$ NW from it and thus this number is also $\epsilon_i$. Therefore we have that $\col_{\lm}(\theta_i, j)=\mu(a,b)$. This shows that $\Phi(\beta_u(D))$ has increments of 1 on $\theta_i$ at column $\mu(a,b)$ compared to $\Phi(D)$. This is equivalent to what the move $\delta_{(i;\mu(a,b))}$ does on $\Phi(D)$. Thus $\Phi(\beta_u(D)) = \delta_{(i;\mu(a,b))}(\Phi(D))$, as desired.
\end{proof}

\begin{figure}
    \centering
    \includegraphics[scale=0.8]{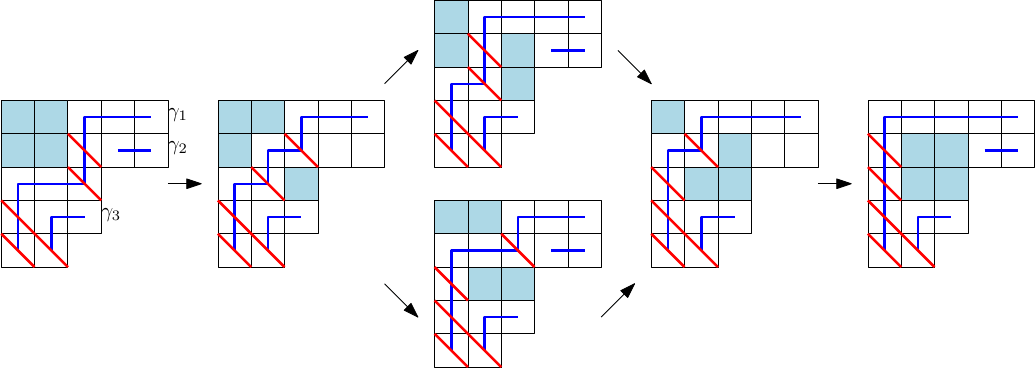}
    \caption{The excited diagrams of shape $55332/22$, their corresponding non-intersecting paths (in \textcolor{blue}{blue}) and broken diagonals (in \textcolor{red}{red}).}
    \label{fig:excited diagrams}
\end{figure}

\begin{figure}
    \centering
    \includegraphics[scale=0.8]{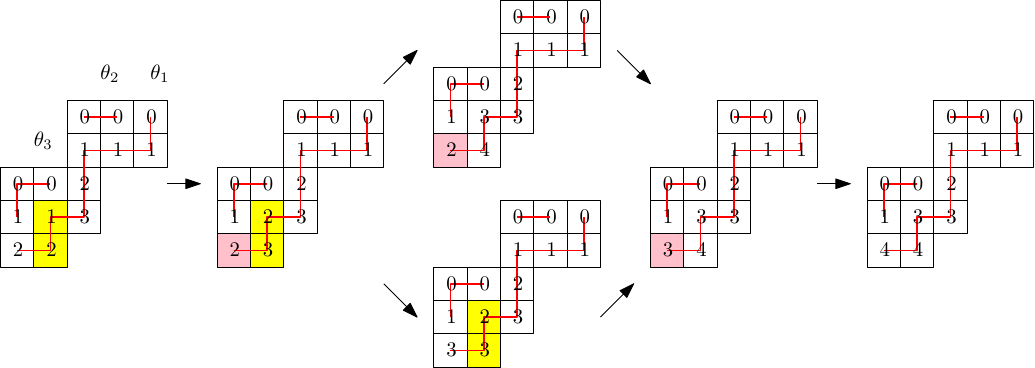}
    \caption{The corresponding minimal SSYT of shape $55332/22$.}
    \label{fig: min SSYT example}
\end{figure}

We are now ready to give the proof of Theorem~\ref{thm: bijection ED and SSYTmin}.

\begin{proof}[Proof of Theorem~\ref{thm: bijection ED and SSYTmin}]
By Lemma~\ref{lemma: beta and delta}, $\Phi$ intertwines the excited moves $\beta$ and $\delta$, as desired. 

Next, by Lemma~\ref{lemma: well-defined and inj}, $\Phi$ is injective. Note that $\Phi([\mu])=T_0$. Given any excited diagram $D$ in $\ED(\lm)$, there exists a sequence $\beta_1,\ldots,\beta_m$ of excited moves such that $D=\beta_m \circ \cdots \circ \beta_1([\mu])$. Iterating Lemma~\ref{lemma: beta and delta} gives that 
\[
\Phi(D) = \delta_m \circ \cdots \circ \delta_1(T_0).
\]
Thus $\Phi(D)$ is in $\SSYT_{\min}(\lm)$ and so $\Phi(\ED(\lm))\subseteq \SSYT_{\min}(\lm)$. It remains to show that $\Phi(D)$ is surjective. Given $T$ in $\SSYT_{\min}(\lm)$, there exists a sequence of $\delta_1,\dots, \delta_k$ of excited moves such that $\delta_k \circ \cdots \circ \delta_1 (T_0) = T$. Again, iterating Lemma~\ref{lemma: beta and delta} we have that for $D'=\beta_k \circ \cdots \circ \beta_1([\mu])$. From this one obtains $\Phi(D')=T$, as desired.  
\end{proof}

\subsection{Non-recursive Characterization of \texorpdfstring{$\SSYT_{\min}(\lambda/\mu)$}{SSYT(lambda/mu)}}

In this section, we give a different non-recursive characterization of the set $\SSYT_{\min}(\lambda/\mu)$ in terms of the height function from Definition~\ref{def:height}. This new characterization will be related to the flagged tableaux $\mathcal{SF}(\lm)$ of the Okounkov--Olshanski formula in Section~\ref{sec: comparison OOF}. Intuitively, the characterization states that a SSYT $T$ of shape $\lm$ is minimal if the values along a strip $\theta_r$ and row $i$ are bounded by the distance from the top in $\theta_k$, and entries in a column of $\theta_r$ differ by $1$.

\begin{theorem} \label{thm: direct char min SSYT}
Given a shape $\lm$, a SSYT $T$ of shape $\lm$ is in $\SSYT_{\min}(\lm)$ if and only if 
\begin{enumerate}[(i)]
    \item For each $(i_r,j_r)$ in $\theta_r$, $T (i_r,j_r) \leq ht_{\theta_r}(i_r)$    \item For any $(i_r,j_r)$ and $(i_r+1,j_r)$ in $\theta_r$, $T (i_r+1,j_r) - T (i_r,j_r) = 1$.
\end{enumerate}
\end{theorem}

In other words, $\SSYT_{\min}(\lambda/\mu)$ consists of SSYT of shape $\lambda/\mu$ where the values along each path $\theta_r$ are bounded by the height $ht_{\theta_r}$ and the values along entries in the same columns of $\theta_r$ differ by one. 

\begin{example}
For $\lambda/\mu = 4321/21$ the SSYT of this shape satisfying Conditions~(i),(ii) of Theorem~\ref{thm: direct char min SSYT} are  
\[
\ytableausetup{smalltableaux}
\ytableaushort{\none\none00,\none01,01,1}* {4,3,2,1}
* [*(cyan)]{2,1,0,0}  \qquad \ytableaushort{\none\none00,\none11,02,1}* {4,3,2,1}
* [*(cyan)]{2,1,0,0}
\qquad \ytableaushort{\none\none00,\none01,11,2}* {4,3,2,1}
* [*(cyan)]{2,1,0,0}
\qquad \ytableaushort{\none\none00,\none11,12,2}* {4,3,2,1}
* [*(cyan)]{2,1,0,0} 
\qquad \ytableaushort{\none\none00,\none11,22,3}* {4,3,2,1}
* [*(cyan)]{2,1,0,0}.  
\]
\end{example}

\begin{proof}[Proof of Theorem~\ref{thm: direct char min SSYT}]
Denote by $\SSYT'_{\min}(\lm)$ the set of SSYT of shape $\lm$ satisfying conditions (i)~(ii) in the statement above. We show that $\SSYT_{\min}(\lm)=\SSYT_{\min}(\lm)$.

First, we use the definition of $\SSYT_{\min}(\lambda/\mu)$ via excited moves $\delta$ to show that $\SSYT_{\min}(\lambda/\mu) \subseteq \SSYT'_{\min}(\lambda/\mu)$. Indeed, $T_0$ is in $\SSYT'(\lm)$ since $T_0$ has consecutive values in its columns and $T_0(i,j) \leq ht_r(i)$. The latter follows from $T_0(i,j)= i-\mu_{j}' +1$ and the fact that the row number $\mathsf{f}(\theta_r)$ of the final element of $\theta_r$ in Definition~\ref{def:height} satisfies $\mathsf{f}(\theta_r) \geq \mu'_j-1$.
Next, we show that the excited move $\delta$ (Definition~\ref{def:delta move}) preserves Conditions (i),(ii) from the statement. For a column $\theta_k(j)$ of $T$ to be active, $T(a,b)< ht_{\theta_k}(i)$ for all $(a,b) \in \theta_k(j)$. The excited move $\delta$ adds one to each entry in the column segment $\theta_k(j)$. Thus we also have $\delta(T)(a,b) \leq ht_{\theta_k}(i)$, preserving Condition~(i). Also, the column differences $\delta(T)(i+1,j)-\delta(T)(i,j)$ in the column segment $\theta_k(j)$ remain $1$, preserving Condition~(ii).

For the other inclusion, $\SSYT'_{\min}(\lambda/\mu) \subseteq \SSYT_{\min}(\lambda/\mu)$, we need to show that given tableau $T\in \SSYT'_{\min}(\lambda/\mu)$, there exists a sequence of excited moves  $\delta^{(1)},\dots, \delta^{(m)}$ such that $\delta^{(m)} \circ \cdots \circ \delta^{(1)} (T_0) = T$.

Let $T$ be a tableau in $\SSYT'_{\min}(\lambda/\mu)$ and $(\theta_1,\dots,\theta_k)$ be the Lascoux--Pragacz decomposition of shape $\lambda/\mu$.  Consider the element-wise difference $\overline{T} := T-T_0$. We know that $T(i,j) \geq T_0(i,j)$ for all $(i,j) \in [\lambda/\mu]$, so $\overline{T}(i,j) \geq 0$ for all cells.
 Choose a strip $\theta_{k_1}$ with maximal diagonal distance $\epsilon_{k_1}$ such that  $\overline{T}$ has a non-zero element. In $\theta_{k_1}$ choose the left most column $c_1$ with a non-zero element. Subtract one from each entry of that column of $\theta_{k_1}$. Repeat the algorithm until $\overline{T} = 0$ (see Figure~\ref{fig:delta sequence}). Then, we have a finite sequence of strips and columns:  $(\theta_{k_1},c_1),\dots, (\theta_{k_m},c_m)$ such that 
 $\delta^{-1}_{(k_m;b_m)}\circ \dots \circ \delta^{-1}_{(k_1;b_1)}(T) = T_0$. 
Thus $\delta^{(m)} \circ \cdots \circ \delta^{(1)}(T_0) = T$, where $\delta^{(j)} := \delta_{(k_{m-j+1};b_{m-j+1})}$ for $j=1,\ldots,m$ and so $T\in \SSYT_{\min}(\lm)$ as desired. 
\end{proof}

\begin{figure}
    \centering
    \includegraphics[scale=0.9]{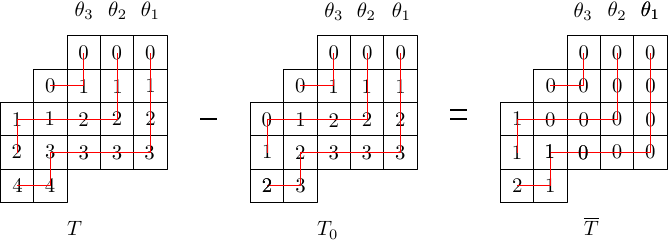}  
    
    \bigskip
    
    \includegraphics[scale=0.9]{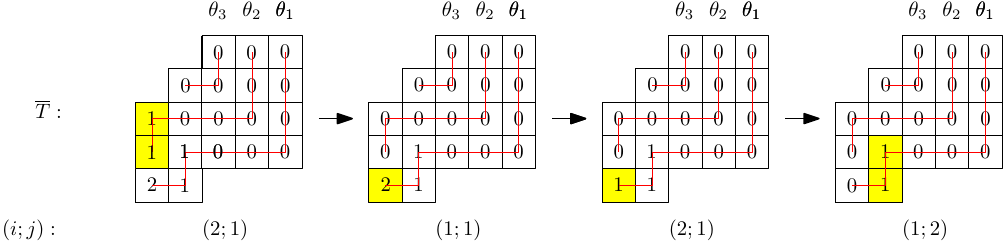}
s    \caption{A sequence of excited moves $\delta_{(k;j)}$ acting on $\overline{T}=T-T_0$.}
    \label{fig:delta sequence}
\end{figure}

\subsection{The inverse of \texorpdfstring{$\Phi$}{Phi}: a bijection between \texorpdfstring{$\SSYT_{\min}(\lm)$}{SSYT(lambda/mu)} and \texorpdfstring{$\E(\lm)$}{E(lambda/mu}}

In this section, we give an explicit description of the inverse of the bijection $\Phi$  that will help us give the reformulation of \eqref{eq:nhlf} in Section~\ref{sec: reformulation NHLF}.

Given any $T\in \SSYT_{\min}(\lm)$, let $\overline{T}= T-T_0$. Note that by Condition~(ii) in Theorem~\ref{thm: direct char min SSYT}, $\overline{T}$ is constant on the column segments of each $\theta_k$. We denote by $\overline{T}(\theta_k(j))$ the value of $\overline{T}$ on the $j$th column segment of $\theta_k$. The next lemma gives a description of the inverse of $\Phi$

\begin{lemma}[Inverse of $\Phi$]\label{lemma: reformulation bijection}
Let $D\in \E(\lm)$ and $T\in \SSYT_{\min}(\lm)$ such that $T = \Phi(D)$. Then 
\[D =  \{(i + \alpha, j+\alpha) \mid (i,j)\in [\mu]\},\]
where $\alpha$ is the number of strips $\theta_k$ such that $\overline{T}(\theta_k(j)) > \mu_j'-i$.
\end{lemma}

\begin{proof}

Given  $T\in \SSYT_{\min}(\lambda/\mu)$, in order to recover its corresponding excited diagram, it suffices to find how many times $\alpha$ has each cell $(i,j)$ in $[\mu]$ been excited to obtain cell $(i+\alpha,j+\alpha)$ in the diagram. 

Let $\overline{T} := T-T_0$. For $(i,j)$ in $[\mu]$, $\overline{T}(\theta_t(j))$ indicates the number of times the $j$th column of $\theta_t$ has been excited by $\delta$. In terms of Kreiman paths, by the proof of Lemma~\ref{lemma: beta and delta}, $\overline{T}(\theta_t(j))$ is also the number of times the $j$th column of $\gamma_t$ has been modified by the ladder move corresponding to $\beta$. Given $(a,b) \in D$, let $(a_\mu,b_\mu)$ be the original position from $[\mu]$ the cell $(a,b)$ came from. The value $\overline{T}(\theta_t(j))$ is the number of cells $(a,b)\in D$ such that $b_\mu = j$ that have been excited  at least $\epsilon_t$ times. This means that $\overline{T}(\theta_t(j))$ many cells in the $j$th column of $[\mu]$ have been excited to obtain $D$ with positions $(a_\mu,j)$, where $\mu'_j - \overline{T}(\theta_t(j)) < a_\mu \leq \mu'_j$. Then for any cell $(a,b)\in D$, where $a_\mu = i$ and $b_\mu = j$, the number of times $\alpha$ the cell has been excited is the total number of $\theta_t$ such that $\overline{T}(\theta_t(j)) > \mu_j' - i$.
\end{proof}

\begin{example} \label{ex: inverse bijection}
Let $T$ be the minimal SSYT of shape $55552/21$ in Figure~\ref{fig: bijection min SSYT to excited diagrams}. We calculate $\overline{T}=T-T_0$. For each $(i,j)$ in $[21]$ we calculate $\alpha$: For $(1,1)$ there is one strip $\theta_k$ such that $\overline{T}(\theta_k(1)) > 2-1$, so $\alpha=1$. For $(1,2)$ there is one strip  $\theta_k$ such that $\overline{T}(\theta_k(1)) > 1-1$, so $\alpha=1$. For $(2,2)$ there are two strips $\theta_k$ such that $\overline{T}(\theta_k(1)) > 1-2$, so $\alpha=2$. Thus $T$ corresponds to the following excited diagram:
\[
D = \{(i+\alpha,j+\alpha) \mid (i,j) \in [21]\} = \{(2,2),(2,3),(4,3)\}.
\]
\end{example}

\section{Reformulation of the NHLF and Special cases of minimal SSYT} \label{sec: reformulation NHLF}

In this section we use the bijection $\Phi$ between minimal SSYT and excited diagrams of shape $\lm$ to reformulate the Naruse hook formula \eqref{eq:nhlf}  in terms of minimal SSYT. We also revisit the enumeration of minimal SSYT for general and particular skew shapes.

\subsection{Reformulation of \texorpdfstring{\eqref{eq:nhlf}}{(NHLF)} in terms of minimal SSYT}

\reformulationNHLFminSSYT

\begin{proof}
The result follows by combining the excited diagram formulation \eqref{eq:nhlf} of Naruse's formula and the description of inverse of $\Phi$ from Lemma~\ref{lemma: reformulation bijection}.
\end{proof}

\begin{figure}
    \centering
    \includegraphics[scale=0.8]{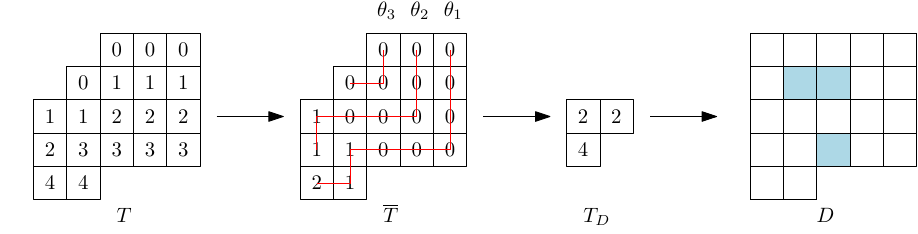}
    \caption{Example of bijection $(\Phi)^{-1}$ between a minimal SSYT and excited diagrams for the shape $\lm=55552/21$.}
    \label{fig: bijection min SSYT to excited diagrams}
\end{figure}

\begin{example} \label{ex: inverse bijection for reformulation}
For the shape $\lm=55332/22$, from each of its minimal tableaux in Figure~\ref{fig: min SSYT example} we find $\overline{T}$:
\[
\ytableausetup{smalltableaux}
\ytableaushort{\none\none000,\none\none000,000,000,00}* {5,5,3,3,2}
* [*(cyan)]{2,2,0,0,0}  \quad \ytableaushort{\none\none000,\none\none000,000,010,01}* {5,5,3,3,2}
* [*(cyan)]{2,2,0,0,0} \quad \ytableaushort{\none\none000,\none\none000,000,010,11}* {5,5,3,3,2}
* [*(cyan)]{2,2,0,0,0} \quad 
\ytableaushort{\none\none000,\none\none000,000,020,02}* {5,5,3,3,2}
* [*(cyan)]{2,2,0,0,0} \quad 
\ytableaushort{\none\none000,\none\none000,000,020,12}* {5,5,3,3,2}
* [*(cyan)]{2,2,0,0,0} \quad 
\ytableaushort{\none\none000,\none\none000,000,020,22}* {5,5,3,3,2}
* [*(cyan)]{2,2,0,0,0}
.  
\]
For each such $\overline{T}$ and $(i,j)$ in $[22]$, we calculate $\textcolor{red}{\alpha}$. Putting everything together using \eqref{eq:reformulate NHLF} gives
\begin{multline*}
f^{\lm} = \frac{14!}{9\cdot 8^2\cdot 7\cdot 6\cdot 5^2\cdot 4^2\cdot 3^2\cdot 2^4}{\big(}
h(1,1)h(1,2)h(2,1)h(2,2)+h(1,1)h(1,2)h(2,1)h(2+\textcolor{red}{1},2+\textcolor{red}{1})\\
+h(1,1)h(1+\textcolor{red}{1},2+\textcolor{red}{1})h(2,1)h(2+\textcolor{red}{1},2+\textcolor{red}{1})+h(1,1)h(1,2)h(2+\textcolor{red}{1},1+\textcolor{red}{1})h(2+\textcolor{red}{1},2+\textcolor{red}{1})+\\+h(1,1)h(1+\textcolor{red}{1},2+\textcolor{red}{1})h(2+\textcolor{red}{1},1+\textcolor{red}{1})h(2+\textcolor{red}{1},2+\textcolor{red}{1})+h(1+\textcolor{red}{1},1+\textcolor{red}{1})h(1+\textcolor{red}{1},2+\textcolor{red}{1})h(2+\textcolor{red}{1},2+\textcolor{red}{1}){\big)},\\
= 445445.
\end{multline*}
\end{example}

\subsection{Enumeration of \texorpdfstring{$\SSYT_{\min}(\lm)$}{SSYT(lambda/mu)}}

Since minimal SSYT are in bijection with excited diagrams, we immediately obtain a determinant formula to count the former.

\begin{corollary}
For a skew shape $\lm$ we have that 
\[
|\SSYT_{\min}(\lm)| \,=\, \det \left[\binom{\vartheta_i + \mu_i -i+j-1}{\vartheta_i-1} \right]_{i,j=1}^{\ell(\mu)},
\]
where $\vartheta_i$ is the row where the diagonal through the cell $(i,\mu_i)$ intersects the boundary of $[\lambda]$. 
\end{corollary}

\begin{proof}
The result follows by combining Theorem~\ref{thm: bijection ED and SSYTmin} and Theorem~\ref{thm:det formula nED}.
\end{proof}

For certain shapes, the characterization of minimal SSYT of skew shape allows to directly enumerate them and recover nice known formulas.

The following formulas follow directly from the equivalence with excited diagrams, together with the fact that the complements of $\E(\la/\mu)$ are non-intersecting lattice paths. Here we give a direct interpretations from the SSYT point of view.

 Let $C_n=\frac{1}{n+1}\binom{2n}{n}$ be the $n$th Catalan number.

\begin{proposition}[{\cite[Cor. 8.1]{MPP2}}] \label{prop:Catalan case}
For the zigzag shape $\lm=\delta_{n+2}/\delta_n$, we have that $$|\SSYT_{\min}(\delta_{n+2}/\delta_n)| = C_n.$$ 
\end{proposition}

\begin{proof}
The zigzag shape $\delta_{n+2}/\delta_n$ has only one Lascoux--Pragacz path. By the characterization of minimal SSYT in Theorem~\ref{thm: direct char min SSYT}, every $T\in \SSYT_{\min}(\delta_{n+2}/\delta_n)$ is determined by the elements of the inner corners. Also, note that the inner corner of the $n$th column is fixed as $0$ because it is always on the first row of the shape. Reading the rest of the inner corners from right to left, we have a word $w_1w_2\dots w_{n-1}$ that determines $T\in \SSYT_{\min}(\delta_{n+2}/\delta_n)$.

Each inner corner $w_i$ is on row $i+1$. Then by condition~(i) from Theorem~\ref{thm: direct char min SSYT} we have $w_i \leq i$ and by condition~(ii), the element below $w_i$ is always $w_i +1$. Thus, $w_{i+1}$ is always to the left of such element, therefore $w_{i+1}\leq w_i + 1$. Words $w_1w_2\cdots w_{n-1}$ satisfying $w_i\leq i$ and $w_{i+1}\leq w_i+1$ are one of the objects counted by the Catalan numbers \cite[Exercise 6.19 u]{EC2}.
\end{proof}

\begin{proposition}[{\cite[Ex. 3.2]{MPP1}}] \label{prop:binomial case}
Let $m$ and $n$ be nonnegative integers. For the reverse hook $\lambda/\mu = (m+1)^{n+1}/m^n$ we have that  
\[|\SSYT_{\min}(\lambda/\mu)| = \binom{n+m}{n}.\]
\end{proposition}

\begin{proof}
The reverse hook shape has only on Lascoux--Pragacz path. By the characterization of minimal SSYT in Theorem~\ref{thm: direct char min SSYT}, in every $T\in \SSYT_{\min}((m+1)^{n+1}/m^n)$  the entries of column $m+1$ are fixed to be $0,1,\ldots,n$. Reading the remaining elements on the $(n+1)$th row  from left to right, gives the sequence $a_1,a_2\dots, a_m$ that determine $T$. By Condition~(ii) in the characterization of minimal SSYT, we have $0\leq a_1\leq a_2 \leq \dots \leq a_m \leq n$. The number of such sequences is given by $\binom{n+m}{n}$, as desired.
\end{proof}

\section{New relations between the Naruse and the Okounkov--Olshanski formula} \label{sec: comparison OOF}

In this section we use the explicit non-recursive description of the minimal tableaux in $\SSYT_{\min}(\lambda/\mu)$ of Theorem~\ref{thm: direct char min SSYT} to relate the number $\nED(\lm)$ of terms of \eqref{eq:nhlf} and the number $\nOOT(\lm)$ of terms of \eqref{eq:OO}.

\subsection{Inequality between the number of terms}

\begin{theorem} \label{thm: ineq ED and OOF}
For a skew shape $\lambda/\mu$ we have that $\nED(\lambda/\mu) \leq \nOOT(\lambda/\mu)$.
\end{theorem}

The result follows immediately from the following lemma that states that after adding one to each entry in a minimal skew SSYT of shape $\lm$, we obtain a flagged tableaux in $\mathcal{SF}(\lm)$. Let  ${\bf 1}_{\lm}$ denote the tableaux of shape $\lm$ with all $1$s. 

\begin{lemma}\label{lemma: SSYT subset SF}
For a skew shape $\lm$ we have that 
\[ \{T+ {\bf 1}_{\lm} \mid T \in \SSYT_{\min}(\lambda/\mu)\} \subseteq \mathcal{SF}(\lambda/\mu),\]
where $T+{\bf 1}_{\lm}$ denotes entrywise addition.
\end{lemma}

\begin{proof}
Let $T$ be a tableau in $\SSYT_{\min}(\lm)$. By Condition~(i) in Theorem~\ref{thm: direct char min SSYT}, we have that for $(i,j)$ in $\theta_r$, $T(i,j) \leq ht_{\theta_r}(i)$. By Definition~\ref{def:height} we have that  $ht_{\theta_r}(i)\leq i-1$. Thus $T+{\bf 1}_{\lm}$ is a flagged tableau in $\mathcal{SF}(\lm)$. 
\end{proof}

\begin{example}
Adding ${\bf 1}_{\lm}$ to the minimal SSYT in Figure~\ref{fig: objects Naruse} 
  yields a flagged tableau in $\mathcal{SF}(\lm)$, which in fact is the one depicted in Figure~\ref{fig: objects OOF}. 
\end{example}

\begin{proof}[Proof of From Theorem~\ref{thm: ineq ED and OOF}]
By Theorem~\ref{thm: bijection ED and SSYTmin}, we have $\nED(\lambda/\mu) = |\SSYT_{\min}(\lambda/\mu)|$. By Lemma~\ref{lemma: SSYT subset SF}, we have $|\SSYT_{\min}(\lambda/\mu)| \leq |\mathcal{SF}(\lm)|$, and the result follows since by Proposition~\ref{bij:OOT and flagged skew tableaux}, we have that $|\mathcal{SF}(\lm)|=\nOOT(\lm)$.
\end{proof}

\begin{remark}
As illustrated in Figure~\ref{fig: objects NHLF and OOF}, if an excited diagram $D$ and a reverse excited diagram $S$ have equal  (up to adding ${\bf 1}_{\lm}$) minimal SSYT and skew flagged tableau of shape $\lm$, their corresponding flagged tableau and Okounkov--Olshanski tableau of shape $\mu$ may differ. 
\end{remark}

The next result characterizes the skew shapes where the Naruse hook-length formula and the Okounkov--Olshanski formula have the same number of terms.

\comparenumberterms

\begin{proof}
The inequality $\nED(\lm) \leq \nOOT(\lm)$ follows from Theorem~\ref{thm: ineq ED and OOF}. By Lemma~\ref{lemma: SSYT subset SF} we have that equality occurs if and only if all the flagged tableaux in $\mathcal{SF}(\lambda/\mu)$ are minimal. First, define the ``maximal'' tableau $T \in \mathcal{SF}(\lambda/\mu)$ as $T(i,j) =i-1$ for $(i,j) \in [\lambda/\mu]$, and let $\theta_1,\ldots,\theta_k$ be the Lascoux--Pragacz decomposition of $\lambda/\mu$. 

Suppose that there is a border strip $\theta_k$ which does not reach the top row of $[\lambda/\mu]$. By the height condition in the definition of minimal tableaux (Condition~(i) in Theorem~\ref{thm: direct char min SSYT}), it follows that for an entry at $(i,j)$ in $\theta_k$ we have $i-1=T(i,j) \leq ht_{\theta_k}(i) <i-1$, a contradiction. Thus all border strips reach the top row of $[\lm]$. 

Next, suppose that there is a border strip $\theta_k$, which has a column $c$ of height at least 2 for some $c\leq \mu_1$, and let $t$ be the topmost row in that column of $\theta_k$. We then construct another tableaux $T'$ from $T$ by setting $T'(i,j)=T(i,j)-1$ for pairs $(i,j)$ such that $i \leq t $ and $j\leq c$, and set $T'(i,j) = T(i,j)$ otherwise. Note that since $j \leq c \leq \mu_1$ we have that $i\geq 2$ and thus $T'(i,j)= T(i,j)-1=i-2$ is well defined. Note also that $T'$ is still a SSYT in $\mathcal{SF}(\lambda/\mu)$. Finally, note that $T'$ violates the condition on consecutive elements in a column of a border strip for the minimal tableaux since $T'(t+1,c) = t > t-1= T'(t,c) +1$ and $(t+1,c), (t,c) \in \theta_k$. Thus $\SSYT_{\min}(\lambda/\mu) \subsetneq \mathcal{SF}(\lambda/\mu)$ and we do not have equality between $\nED(\lm)$ and $\nOOT(\lm)$. 

Hence we see that the border strips $\theta_i$ are entirely horizontal in the first $\mu_1$ columns of $\lambda/\mu$ and we have $d-r$ of them starting in the first $\mu_1$ columns. Thus $\lambda_d \geq \mu_1$. Also, since these border strips reach the top row, they must all cross the diagonal starting at $(r,\mu_r)$ and thus this diagonal extends to the last row of the shape, so $\lambda_d \geq \mu_r + d-r$, which proves the forward direction. See Figure~\ref{fig:illustration equality terms}.

Now, if $\lambda_d \geq \mu_r+d-r$, then it is easy to see that all border strips  $\theta_i$ reach the top row and they are horizontal in the first $\mu_1$ columns. The tableaux in $\mathcal{SF}(\lambda/\mu)$ consisting of the columns past $\mu_1$ are of straight shape, and since the entries in row $i$ are bounded by $i-1$, the strict increase in the full columns implies that such entries have to be exactly $i-1$. For the first $\mu_1$ columns, since all border strips are horizontal, there are no further restrictions coming from the minimal tableaux. Thus $\SSYT_{\min}(\lambda/\mu)=\mathcal{SF}(\lambda/\mu)$, as desired.
\end{proof}

Motivated by this result, we call the connected skew shapes $\lm$ where we have equality $\nED(\lm)=\nOOT(\lm)$ {\em slim skew shapes}, see \cite[\S 11]{MPP4} and \cite[\S 2.1]{MPP3}\footnote{The convention for slim shapes in the articles \cite[\S 11]{MPP4} and \cite[\S 2.1]{MPP3} matches our condition in the case of  $r=1$.}. Next, we show that for slim skew shapes $\lm$ where there is equality $\ED(\lm)=\nOOT(\lm)$, this number is given by a product.

\begin{corollary} \label{cor: prod formula when ED=OOT}
For a connected skew shape $\lambda/\mu$ satisfying $\lambda_d\geq \mu_r+d-r$ where $d=\ell(\lambda)$ and $r=\max\{i\mid \mu_1=\mu_i\}$, we have that 
\[
\nED(\lm)=\nOOT(\lm) =  \prod_{u \in [\mu]} \frac{d+c(u)}{h(u)}.
\]
\end{corollary}

\begin{proof}
By Theorem~\ref{thm: num ED vs OOF}, for such skew shapes $\lm$ we have the equality $\nED(\lm)=\nOOT(\lm)$. By definition of $r$ we have that $\mu_1=\mu_r$, and thus $\lambda_d \geq \mu_1 + d-r \geq \mu_1$. By \cite[Proposition 7.14]{MZ}, since $\mu_1 \leq \lambda_d$ then $\mathcal{OOT}(\lm)=\SSYT(\mu,d)$ and so $\nOOT(\lm)=s_{\mu}(1^d)$, which is given by the {\em hook-content formula} (see \cite[Cor. 7.21.4]{EC2} and \cite[Cor. 7.15]{MZ}). 
\end{proof}

\begin{figure}
    \centering
    \includegraphics[scale=0.8]{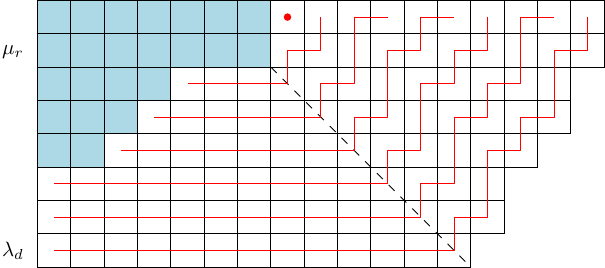}
    \caption{Illustration of skew shapes $\lambda/\mu$ and their Lascoux--Pragacz decomposition with $\lambda_d\geq \mu_r+d-r$ where $\nED(\lambda/\mu)=\nOOT(\lambda/\mu)$.}
    \label{fig:illustration equality terms}
\end{figure}

\subsection{Comparing terms of hook formulas for skew shapes}

A natural question to ask is whether for slim shapes $\lm$, the formulas \eqref{eq:nhlf} and \eqref{eq:OO} not only have the same number of terms but they are also term-by-term equal. As the next example shows, the terms can be different. 

\begin{example}
The shape $\lm=322/11$ is slim with three excited and three reverse excited diagrams, respectively. See Figure~\ref{fig:excited_compare},\ref{fig:oof_compare}. The formulas \eqref{eq:nhlf} and \eqref{eq:OO}, respectively give
\begin{align*}
f^{\lm} &= \frac{5!}{5\cdot 4 \cdot 3 \cdot 2^2\cdot 1}\left(5\cdot 3 + 5 \cdot 1 + 2\cdot 1\right),\\
&= \frac{5!}{5\cdot 4 \cdot 3 \cdot 2^2\cdot 1}\left(4\cdot 2 + 4 \cdot 2 + 3 \cdot 2 \right) = 11.
\end{align*}
\end{example}

\begin{figure}

\begin{subfigure}[b]{0.3\textwidth}
     \centering
      \includegraphics[scale=0.5]{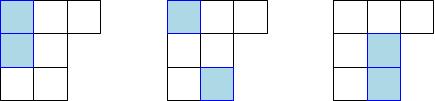} 
     \caption{}
     \label{fig:excited_compare}
     \end{subfigure}
\begin{subfigure}[b]{0.3\textwidth}
     \centering    
   \includegraphics[scale=0.5]{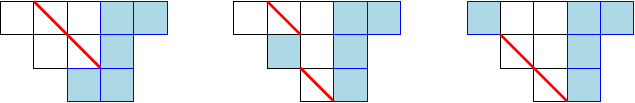} 
   \caption{}
   \label{fig:oof_compare}
   \end{subfigure}
\begin{subfigure}[b]{0.33\textwidth}
     \centering  
      \includegraphics[scale=0.5]{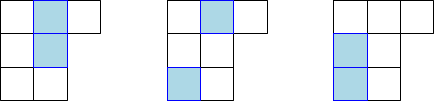}
   \caption{}
   \label{fig:ne_excited_compare}
   \end{subfigure}
\caption{For the shape slim shape $\lambda/\mu=322/11$, an illustration of its (a) Excited diagrams, (b) reverse excited diagrams, and (c) NE excited diagrams.}
\end{figure}

However, there is a reformulation of \eqref{eq:nhlf} from \cite[\S 3.4]{MPP3} that is then term-by-term equal to \eqref{eq:OO} for slim shapes. To see this,  we need the following variation of excited diagrams from \cite[\S 3.4]{MPP3}. Given a slim skew shape $\lm$ with  $\mu=(\mu_1,\ldots, \mu_{\ell})$, we  denote by $\widehat{\mu}=(0^{d-\ell},\mu_{\ell},\mu_{\ell-1},\ldots,\mu_1)$ the horizontal flip of $\mu$. Since $\lm$ is slim then the Young diagram $[\widehat{\mu}]\subset [\lm]$. A \emph{NE-excited diagram} is a subdiagram of $[\lambda]$ obtained from the diagram $[\widehat{\mu}]$ after a sequence of moves $D\mapsto D'$ where $(i,j)$ is replaced by $(i-1,j+1)$ if $(i,j)$ is in the diagram $D$ and all of $(i-1,j), (i-1,j+1), (i,j+1)$ are in $[\lambda]\setminus D$. See Figure~\ref{fig:ne_excited_compare} for an example. The set of these diagrams is denoted by $\ED^{\nearrow}(\lm)$. From \cite[\S 3.4]{MPP3}, flipping horizontally the diagrams gives a bijection between $\ED(\lm)$ and $\ED^{\nearrow}(\lm)$. Moreover, as a consequence of the multivariate identity in \cite[Thm. 3.12]{MPP3} there is a Naruse-type formula for $f^{\lm}$ in terms of the diagrams in $\ED^{\nearrow}(\lm)$ that appears in \cite[\S 9.1]{P}.

\begin{corollary}[{\cite[\S 3.4]{MPP3},\cite[\S 9.1]{P}}]
For a slim skew shape $\lm$ of size $n$ we have that 
\begin{equation} \label{eq:NE_NHLF_slim}
f^{\lm} \,=\,  \frac{n!}{\prod_{u\in [\lambda]} h(u)} \sum_{D \in \ED^{\nearrow}(\lm)} \prod_{(i,j)\in D} h(i,j).
\end{equation}
\end{corollary}

Next, we show that for slim shapes $\lm$, there is bijection between the reverse excited diagrams and the NE-excited diagrams of $\lm$ that preserve the weights of \eqref{eq:OO} and \eqref{eq:NE_NHLF_slim}, respectively.

\begin{definition}\label{def:map Psi}
Given an excited diagram $D$ in $\ED^{\nearrow}(\lm)$, let $T:=\Psi(D)$ be the tableau  $\mu$ where $T(x,y)=d+1-i_x$ where $(i_x,j_y)$ is the cell in $D$ corresponding to $(d+1-x,y)$ in $[\widehat{\mu}]$. That is $i_x$ is the row where the cell $(d+1-x,y)\in [\widehat{\mu}]$ ended up in after the excited moves. See Figure~\ref{fig: bijection excited rev excited slim shapes} for an example.
\end{definition}

\begin{proposition} \label{prop:weight preserving bijection between NEED and OOT}
For a slim shape $\lm$, the map $\Psi$ defined above is a weight preserving bijection between  $\ED^{\nearrow}(\lm)$ and $\mathcal{OOT}(\lm)$. That is, if $\Psi(D)=T$, then $$\prod_{(i,j)\in D} h(i,j)=\prod_{(i,j)\in [\mu]} (\lambda_{d+1-T(i,j)} +i-j).$$ 
\end{proposition}

\begin{proof}
We need to prove that $\Psi$ is a well-defined map between $\ED^{\nearrow}(\lm)$ and $\mathcal{OOT}(\lm)$. For $D$ in $\ED^{\nearrow}(\lm)$, the tableau $T:=\Psi(D)$ of shape $\mu$ is semistandard by a similar argument as the one in the proof of \cite[Prop. 3.6]{MPP1} (induction on the number of excited moves). Also by construction, the entries in $T$ are bounded by $d$ and so $T\in \SSYT(\mu,d)$. Since $\lm$ is slim, by the argument in the proof of Corollary~\ref{cor: prod formula when ED=OOT} we have that $\mathcal{OOT}(\lm)=\SSYT(\mu,d)$. Thus $\Psi$ is well defined. 

Next, we show that $\Psi$ is bijective by building its inverse. Given $T$ in $\mathcal{OOT}(\lm)$, let $D:=\Omega(T)$ be the set $\{(d+1-T(x,y), y-x+T(x,y)) \mid (x,y) \in [\mu]\}$. This map is well-defined by an argument similar to the one in the proof of  \cite[Prop. 3.6]{MPP1}. By construction, we have that $\Omega=\Psi^{-1}$, as desired.

Finally, we show that $\Psi$ is weight preserving. If $\Psi(D)=T$ with $T(x,y)=d+1-i_x$ where $(i_x,j_y)$ is the cell in $D$ corresponding to the cell $(d+1-x,y)$ in $[\widehat{\mu}]$, then $i_x=d+1-x-c$, and $j_y=y+c$ for some nonnegative integer $c$ (how many times the cell $(d+1-x,y)$ was excited to obtain $(i_x,j_y)$). Since $\lm$ is slim, we have that $\lambda'_{j_x}=d$.  Thus the hooklength of cell $(i_x,j_y)$ equals 
\[
h(i_x,j_y)=\lambda_{i_x}-i_x+\lambda'_{j_y}-j_y+1= \lambda_{d+1-T(x,y)} +x-y,
\]
as desired.
\end{proof}

\begin{example} \label{ex: weight preserving bijection for slim shapes}
Figure~\ref{fig: bijection excited rev excited slim shapes} illustrates an excited diagram $D$ in $\ED^{\nearrow}(\lm)$ for the slim shape $\lm=775/21$, its corresponding tableau $T=\Psi(D)$ in $\mathcal{OOT}(\lm)$, and  the corresponding reverse excited diagram $D'$ from $T$. The weight of the hook-lengths of $D$ is $8\cdot 6\cdot 5$ which matches the weight of the broken diagonals of $D'$. 
\end{example}

\begin{corollary}
For a slim skew shape $\lm$ of size $n$ we have that the variation of the NHLF in \eqref{eq:NE_NHLF_slim} is term-by-term equal to \eqref{eq:OO}.
\end{corollary}

\begin{proof}
    This follows immediately by applying the bijection $\Psi$, which is weight preserving by Proposition~\ref{prop:weight preserving bijection between NEED and OOT}, to rewrite \eqref{eq:NE_NHLF_slim} in terms of Okounkov--Olshanski tableaux to obtain \eqref{eq:OO}.  
\end{proof}

\begin{figure}
\begin{center}
\includegraphics[scale=0.8]{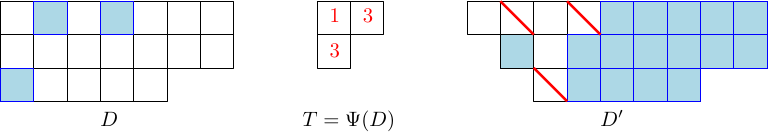}
    \caption{Example of the weight preserving bijection  between an excited diagrams in $\ED^{\nearrow}(\lm)$ (left) and a reverse excited diagrams in $\mathcal{RED}(\lm)$ (right) for a slim shape via tableaux in $\mathcal{OOT}(\lm)$ (middle).}
\label{fig: bijection excited rev excited slim shapes}
\end{center}
\end{figure}


\section{Hillman-Grassl bijection on minimal skew tableaux}\label{sec: HG bijection on min SSYT}

In this section we show that the Hillman--Grassl bijection restricted to the minimal SSYT of shape $\lm$ is equivalent to the inverse of bijection $\Phi$ from Section~\ref{sec: minimal SSYT}. Recall the definition of the array $A_D$ for an excited diagram $D$ in $\ED(\lm)$ and of the arrays $\mathcal{A}^*(\lm)$ from Section~\ref{sec: HG background}.

\subsection{Equivalence of Hillman--Grassl map and \texorpdfstring{$\Phi^{-1}$}{Phi inverse} on minimal SSYT}

We state the main result of this section.

\PhivsHG

To prove this result we will need the following two lemmas that use the known connection between the Hillman--Grassl correspondence and the RSK bijection (see Theorem~\ref{thm: Greene's theorem for HG}). Let $A({\gamma_i(D)})$ be the a subarray of $A_D$ of shape $\lambda \setminus \bigcup_{j=1}^{i-1} \gamma_j(D)$ with support on the broken diagonal of $\gamma_i(D)$.

\begin{lemma} \label{lemma: Greene's theorem for A_Ds}
For $D\in \ED(\lm)$ and $k$  an integer $1-\ell(\lambda) \leq k \leq \lambda_1-1$, let $T:=\HG^{-1}(A_D)$ and $\nu=\nu^{(k)}$ be the partition whose parts are the entries in the $k$-diagonal of $T$. Then $\nu_i=|A({\gamma_i(D)})_k|$, that is $\nu_i$ is the number of broken diagonals in $\gamma_i(D)\cap \square^{\lambda}_k$.
\end{lemma}

\begin{proof}
Since $A_D=\HG(T)$, by Theorem~\ref{thm: Greene's theorem for HG} it suffices to show that $ac_t(A_k) = a_1+a_2+\cdots + a_t$, where $a_i = |A({\gamma_i(D)})_k|$. We consider only $\gamma_i(D)$'s which intersect $\square^\lambda_k$ and index them in order from the NW-most to the SE-most within $\square^{\lambda}_k$ (we can order them, because every such path will intersect the diagonal $j-i=k$). 

First, we note that $a_1 >a_2 >\cdots$: this is because $a_i$ is equal to the height of $\gamma_i(D)$ minus 1, and the paths within $\square^{\lambda}_k$ start along the lower horizontal border at points $P_i$ and end on the right vertical border at points $Q_i$, so they end below each other. 

Next, note that $ac_t(A_k) \geq a_1+\cdots +a_t$, because the 1s in each $\gamma_i(D)$ form a NE path (i.e. increasing sequence).

To show the opposite inequality, we consider the broken diagonals in $\square^{\lambda}_k$, that is $\Br(D)\cap \square^{\lambda}_k$. Let $\alpha_1 \sqcup \cdots \sqcup \alpha_t$ be a disjoint collection of maximal nonoverlapping NE paths, so $ac_t(A_k) = |\alpha_1|+\cdots+|\alpha_t|$. We can assume that these paths are nonintersecting also, as we can relabel their parts at  intersections. The diagonal steps in each broken diagonal form a SE monotone path, i.e. no squares in the same row or column. Thus for every broken diagonal $d_i$, we must have $|d_i \cap \alpha_j| \leq 1$.  Thus, 
$$ac_t(A_k) \leq \sum_i \max\{ |d_i \cap \square^{\lambda}_k |,t\}.$$
We now show that this upper bound is exactly the contribution from the $\gamma(D)$'s. We have that $d_i \cap \square^{\lambda}_k$ consists of the vertical steps of the $\gamma(D)$ paths which intersect that broken diagonal. As the $\gamma$ paths are greedily selected from the inner shape outwards, the first paths will contain the first (NW-most) points from the corresponding $d_i$'s. Thus
$$|\{ d_i \cap \square^{\lambda}_k \}| = \max\{ j: \gamma_j(D) \cap d_i \cap \square^{\lambda}_k\}, $$
and since we are looking at the maximum $t$ many, we naturally limit ourselves to $j \leq t$. Thus
 we have
$$ac_t(A_k) \leq \sum_i \max\{ |d_i \cap \square^{\lambda}_k |,t\} = |\gamma_1(D) \cap \square^{\lambda}_k|+\cdots+|\gamma_t(D) \cap \square^{\lambda}_k| =a_1+\cdots+a_t,$$
and the proof is done.

\end{proof}

The next lemma confirms that counting broken diagonals in the rectangular subarray $(A_D)_k$ is another way of describing the bijection $\Phi$.

\begin{lemma} \label{lemma: Greene's theorem is Phi}
For $D\in \ED(\lm)$ and  $k$ an integer $1-\ell(\lambda) \leq k \leq \lambda_1-1$, let $T':=\Phi(D) \in \SSYT_{\min}(\lm)$ and and $\nu=\nu^{(k)}$ be the partition whose parts are the entries in the $k$-diagonal of $T'$. Then $\nu_i=|A({\gamma_i(D)})_k|$, that is $\nu_i$ is the number of broken diagonals in $\gamma_i(D)\cap \square^{\lambda}_k$.
\end{lemma}

\begin{proof}
First, recall that $T'$ is decomposed into Lascoux--Pragacz paths $\theta_i,\dots, \theta_m$, where $\theta_1$ is the of $\lm$ wrapping around $\la$. Then for the $k$-diagonal of $T'$ with values $\nu_1,\dots, \nu_t$, the cell with the value $\nu_i$ is contained in $\theta_{i'}$ for some $i'\geq i$, where $\epsilon_{i'} = i-1$ and the cell $(u_1,v_1)\in \theta_1$ with value $\nu_1$ is always the most outer element of the $k$-diagonal. Note that the cell $(u_1,v_1)$ is also the most outer corner of the corresponding $\square^{\lambda}_k$ in $A_D$. Also, as stated in Definition~\ref{def: phi* bijection},  the $j$th column of $\gamma_{i'}(D)$ corresponds to the $j$th column of $\theta_{i'}$, starting from the left-most column. Denote $t_i$ and $s_i$ the starting column and row of $\theta_{i'}$ respectively. Let $(u,v)$ be the cell with value $\nu_i$, which is on the $j$th column of $\theta_{i'}$, i.e. column $t_i + j-1$. Then by the construction of $\Phi$ in Definition~\ref{def: phi* bijection} (1) and (2), we know that $\nu_i= T'(u,v)$ is the number of accumulated broken diagonals of $\gamma_i(D)$ until the $j$th column of this path and by Definition~\ref{def: phi* bijection}~(3) we need to subtract the number $s_i-u$ of north steps from the initial cell, thus 
\begin{equation} \label{eq: formula nui}
\nu_i = |A(\gamma_i(D))|_j-s_i+u,
\end{equation}
where $|A(\gamma_i'(D))|_j$ is the number of broken diagonals of $\gamma_{i'}(D)$ until the $jth$ column.
For the Kreiman paths in $A_D$, denote by $g_i$ and $f_i$ the starting column and row of $\gamma_{i'}(D)$, respectively. The $j$th column of $\gamma_{i'}(D)$ is then $g_i + j -1$. By Lemma~\ref{lemma: gamma-theta} we know that $g_i-t_i = \epsilon_{i'}$, the diagonal distance from $\theta_1$ to $\theta_i$ (and similarly, also $f_i-s_i = \epsilon_{i'}$). So the $j$th column of $\gamma_i(D)$ is $t_i+j-1+\epsilon_i$. Recall that the cell with value $\nu_i$ of the $k$-diagonal is on column $t_i+j-1$. Then the column $t_i+j-1+\epsilon_{i'}$, the $j$th column of $\gamma_i(D)$, is the last column in $\square^{\lambda}_k$. Also the rectangle $\square_k^{\lambda}$  does not include the broken diagonals of $\gamma_i(D)$ that lie below the position $(u_1,v_1)$, so we need to subtract $f_i-u_1$ many broken diagonals from $|A(\gamma_i(D))|$ to obtain $|A(\gamma_i(D))_k|$. That is, we have that 
\begin{equation} \label{eq formula A(gammai)_k}
    |A(\gamma_i(D))_k| = |A(\gamma_i(D))|_j - f_i-u_1.
\end{equation}
Note that the diagonal distance between the cell with the value $\nu_i$ and $\nu_1$ is $\epsilon_i$. Then $u_1 - u= \epsilon_i$. Then by Lemma~\ref{lemma: gamma-theta}, $f_i- u_1 = s_i + \epsilon_i - (u - \epsilon_i) =s_i - u$. Combining this with \eqref{eq: formula nui} and \eqref{eq formula A(gammai)_k} we obtain that $\nu_i=|A({\gamma_i(D)})_k|$, as desired.  
\end{proof}

We are now ready to prove the main result of this section.

\begin{proof}[Proof of Theorem~\ref{thm: phi vs hg}]
For an excited diagram $D\in \ED(\lm)$, let $T:=\HG^{-1}(A_D)$ and $T'=\Phi(D)$. By Lemma~\ref{lemma: Greene's theorem for A_Ds} and Lemma~\ref{lemma: Greene's theorem is Phi}, the tableaux $T$ and $T'$ agree on each $k$-diagonal for $1-\ell(\lambda) \leq k \leq \lambda_1-1$. Thus $T=T'$ as desired.
\end{proof}

\begin{example} \label{ex:GreeneHGPhi}
Figure~\ref{fig:exGreenesHGPhi} shows an excited diagram $D$ for the shape $\lm=55552/21$ with its broken diagonals and Kreiman paths $\gamma_1,\gamma_2,\gamma_3$ and the tableau $T=\HG^{-1}(A_D)=\Phi(D)$. The parts of the partition $\nu=\nu^{(1)}=321$ obtained from the $1$-diagonal of $T$ can be read from the number of broken diagonals in $\gamma_i(D)$ inside the rectangle $\square_1^{\lambda}$ for $i=1,2,3$, respectively.

\end{example}

\begin{remark}
Lemma~\ref{lemma: Greene's theorem for A_Ds} does not hold for general excited arrays with nonnegative entries other than $0/1$. This is because the longest ascending chain might involve entries from distinct $\gamma_i$. See Example~\ref{ex:counterexGreenesHGPhi}.
\end{remark}

\begin{example}  \label{ex:counterexGreenesHGPhi}
Figure~\ref{fig:counterexGreenesHGPhi} shows an excited array $A'_D$ for the same excited diagram as Example~\ref{ex:GreeneHGPhi} and the tableau $T'=\HG^{-1}(A'_D)$. The first part $5$ of the partition $\nu=\nu^{(1)}=521$ obtained from the $1$-diagonal of $T'$ corresponds to an ascending chain involving entries from two paths $\gamma_1$ and $\gamma_2$.
\end{example}

\begin{figure}
 \begin{subfigure}[normal]{0.4\textwidth}
     \centering
    \includegraphics[scale=0.7]{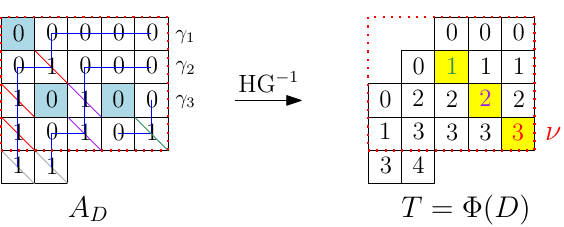}
    \caption{}
    \label{fig:exGreenesHGPhi}
\end{subfigure}
\quad 
\begin{subfigure}[normal]{0.4\textwidth}
\centering
\raisebox{13pt}{
    \includegraphics[scale=0.7]{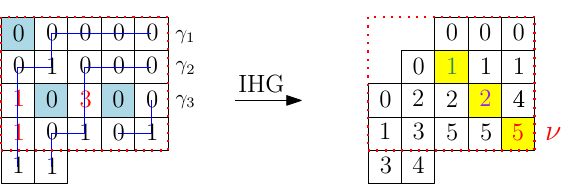}}
    \caption{}
    \label{fig:counterexGreenesHGPhi}
\end{subfigure}    
    \caption{(a) Illustration of the relation between Greene's theorem and the map $\Phi$. The parts of $\nu$ encode the number of broken diagonals in each $\gamma_i(D)$ inside the rectangle $\square^{\lambda}_k$. (b) Counterexample of why this relation does not hold if the entries of the array are other than $0/1$.}
    \label{fig:exGreenes}
\end{figure}

\subsection{Leading terms of the SSYT \texorpdfstring{$q$}{q}-analogue of Naruse's formula} \label{sec:qnumED}

In this section we give an expression for the leading terms of \eqref{eq:qnhlf}.

\leadtermsqNHLF

\begin{proof}
From Theorem~\ref{thm: phi vs hg} and the fact that the Hillman--Grassl bijection is weight preserving \eqref{eq: HG hook weight} we have that 
\[
\sum_{T \in \SSYT_{\min}(\lm)} q^{|T|} \,=\, \sum_{D \in \ED(\lm)} q^{\sum_{u \in A_D} h(u)}.
\]
By \cite[Prop. 7.16]{MPP1} we have that $\sum_{u \in A_D} h(u) = \sum_{(i,j) \in \lambda\setminus D} (\lambda'_j-i)$ and the result follows.
\end{proof}

This result motivates calling the polynomial in \eqref{eq:lead term qnhlf} a $q$-analogue of the number of excited diagrams of shape $\lm$ and denoting it by $\mathcal{E}_q(\lm)$. Next, we give some examples of this polynomial. It would be of interest to study $\mathcal{E}_q(\lm)$ further.

\begin{example}
For the shape $\lm=55332/22$ whose excited diagrams and minimal tableaux are illustrated in Figures~\ref{fig: min SSYT example},\ref{fig:excited diagrams} we have that 
\[
\mathcal{E}_q(\lm) \,=\, q^{14}(q^2+q^3+q^4+q^5+q^6).
\]
\end{example}

\begin{example}
Let $m$ and $n$ be nonnegative integers. For the reverse hook $\lambda/\mu = (m+1)^{n+1}/m^n$, the number of excited diagrams is given by $\binom{n+m}{n}$ (Proposition~\ref{prop:binomial case}) and these diagrams are in bijection with lattice paths $\gamma$ from $(m+1,1)$ to $(1,n+1)$ inside the $(m+1)\times(n+1)$ rectangle (see \cite[Ex. 3.2]{MPP1}).  A direct consequence of \cite[Ex. 3.2]{MPP1} and \cite[Prop. 1.7.3]{EC1} is the following
\[
\mathcal{E}_q((m+1)^{n+1}/m^n)  \,=\, q^{\binom{n+1}{2}} \sum_{\gamma:(m+1,1)\to(1,n+1)}  q^{\textup{area}(\gamma)} \,=\, q^{\binom{n+1}{2}} \qbinom{n+m}{n}, 
\]
where $\textup{area}(\gamma)$ is the number of cells in the rectangle $(m+1)\times (n+1)$ South East of $\gamma$, and $\qbinom{a}{b}$ denotes the {\em $q$-binomial coefficient}.
\end{example}

\begin{example}
For the reverse hook $\lm=\delta_{n+2}/\delta_n$, the number of excited diagrams is given by the Catalan number $C_n$ (Proposition~\ref{prop:Catalan case}) and these diagrams are in bijection with Dyck paths (see \cite[\S 8]{MPP2}).  Let $\textup{Dyck}(n)$ be the set of Dyck paths $\mathsf{p}$ of size $n$, a direct consequence of \cite[Cor. 1.7]{MPP2}  is the following
\begin{align*}
\mathcal{E}_q(\delta_{n+2}/\delta_n)  \,=\, \sum_{\mathsf{p} \in \textup{Dyck}(n)} q^{\sum_{(a,b)\in \mathsf{p}} b} 
\,=\, q^n \sum_{\mathsf{p} \in \textup{Dyck}(n)} q^{2\cdot \textup{area}(\mathsf{p})} \,=\, q^n C_n(q^2),
\end{align*}
where $\textup{area}(\mathsf{p})$ denotes the area of the Dyck path $\mathsf{p}$ and $C_n(q)$ denotes the {\em $q$-Catalan number} (see \cite[Ch. 1]{Hagqtcat}).
\end{example}

\subsection{Additivity consequences}

In this section as a consequence of Theorem~\ref{thm: phi vs hg}, we derive some linearity properties of the Hillman--Grassl correspondence on minimal SSYT. 

Given $T$ in $\SSYT(\lm)$, let $T_{\theta_i}$ be the tableau of shape $\lambda \setminus \bigcup_{j=1}^{i-1} \theta_j$ with entries of $T$  in $\theta_i$. We  define the sum  $A_{\gamma_i(D)}+A_{\gamma_{i+1}(D)}$ as an array of the same shape as $A_{\gamma_i(D)}$ with supports on the broken diagonals of $\gamma_i(D)$ and $\gamma_{i+1}(D)$. In other words the array value of $(u,v)\in  A_{\gamma_i} + A_{\gamma_{i+1}}$ is $A_{\gamma_{i}}(u,v) + A_{\gamma_{i+1}}(u-t, v-t)$ where $t = \epsilon_{i+1} - \epsilon_{i}$.

\begin{corollary}\label{cor: union everything}
Let $D \in \ED(\lm)$ and $T=\Phi(D) \in \SSYT_{\min}(\lm)$. Then for $i=1,\ldots,k$, we have that
\[\sum_{i=1}^k HG(T_{\theta_i}) = HG(T).\]
\end{corollary}

\begin{proof}
As mentioned above, we have that $\sum_{i=1}^k A_{\gamma_i(D)}  = A_D$ and that the tableau $T$ decomposes as a disjoint union of $\bigcup_i T_{\theta_i}$ of tableaux $T_{\theta_i}$ in each border strip. By Theorem~\ref{thm: phi vs hg} we have that  $\HG(T)=A_D$. By construction we have that $T=\bigcup_i T_{\theta_i}$. Since by Definition~\ref{def: phi* bijection}, the bijection $\Phi$ builds $T_{\theta_i}$ just from information in $\gamma_i(D)$ then $\Phi(A_{\gamma_i(D)})=T_{\theta_i}$. By Theorem~\ref{thm: phi vs hg} applied to the border strip $\theta_i$ we then have that $\HG(T_{\theta_i})=A_{\gamma_i(D)}$ and the result follows.
\end{proof}

\begin{example}
Consider the skew shape $\lm=55552/21$ with Lascoux--Pragacz decomposition $\theta_1,\theta_2,\theta_3$ and the excited diagram $D$ in $\ED(\lm)$ with Kreiman decomposition $\gamma_1,\gamma_2, \gamma_3$ at the bottom of Figure~\ref{fig:HG theta}. For $T=\Phi(D)$ we have that $\HG(T_{\theta_1})+\HG(T_{\theta_2})+\HG(T_{\theta_3})=\HG(T)$ as illustrated in the figure. 
\end{example}

We can also view Corollary~\ref{cor: union everything} as a statement about a rearrangement of the lattice paths involved in the inverse Hillman--Grassl map for each broken diagonal in $A_{\gamma_i(D)}$ and each broken diagonal in $A_D$.  To state this result, we introduce some notation. For each broken diagonal support $b\in A_{\gamma_i(D)}$ ($b\in A_{D}$ respectively), denote the collection of cells in the lattice path built for its inverse Hillman-Grassl step by $L(b;\theta_i)$ ($L(b)$ respectively).

\begin{corollary}
\label{lem: additivity of k paths}
Let $D \in \ED(\lm)$ be an excited diagram with Kreiman decomposition $(\gamma_1,\ldots,\gamma_k)$, then as  multisets we have that 
\begin{equation}\label{eq:equality sets lattice paths}
    \bigcup_{i=1}^k \Biggl(\bigcup_{b\in \Br(\gamma_i)} L(b;\theta_i) \Biggr)   = \bigcup_{b\in \Br(D)}  L(b).
\end{equation}

\end{corollary}

\begin{figure}
    \centering
    \includegraphics[scale=0.8]{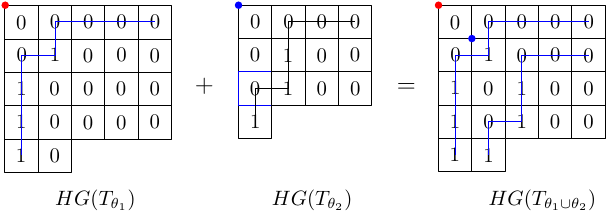}
    \caption{Convention for adding arrays $A_{\gamma_1(D)}$ and $A_{\gamma_2(D)}$.}
    \label{fig:hg additive}
\end{figure}

\begin{proof}
By Theorem~\ref{thm: phi vs hg} we have that $HG^{-1}(A_D) = T$ and as explained in the proof of Corollary~\ref{cor: union everything}, from Theorem~\ref{thm: phi vs hg} and Definition~\ref{def: phi* bijection} we also have that $HG^{-1}(A_{\gamma_i(D)})=T_{\theta_i} $. Then for each $(r,s)\in \theta_i$, the entry $T(r,s)$ is the multiplicity of $(r,s)$ in the multiset in the RHS of \eqref{eq:equality sets lattice paths} and this entry $T(r,s)=T_{\theta_i}(r,s)$ is also the multiplicity of $(r,s)$ in the multiset $\bigcup_{b \in \Br(\gamma_i)} L(b;\theta_i)$ in the LHS of \eqref{eq:equality sets lattice paths}. The equality of this multisets follows.
\end{proof}

\begin{example}
See Figure~\ref{fig:HG combined and separate} for an example of the equality of multisets in \eqref{eq:equality sets lattice paths}. For instance, the entries in the paths in the first column highlighted in red equal the entries in the second and third column highlighted in blue.
\end{example}

\begin{figure}
    \centering
    \includegraphics[scale=0.8]{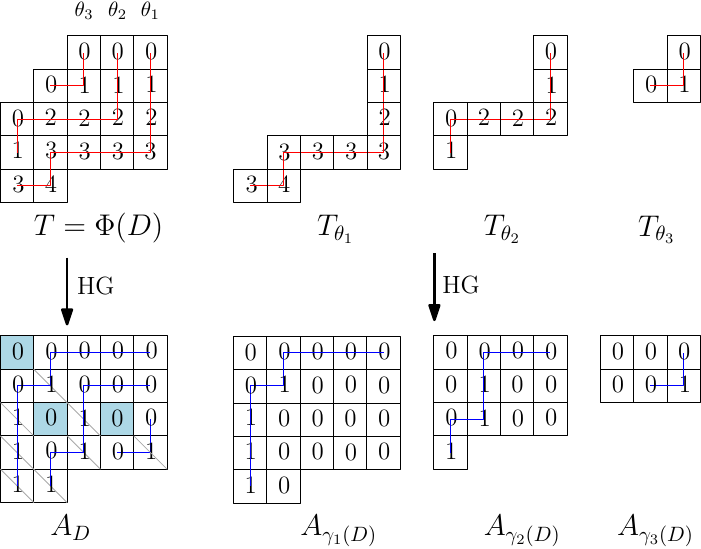}
    \caption{The Hillman-Grassl map of $T$, and each of $T_{\theta_1}$, $T_{\theta_2}$, and $T_{\theta_3}$ separately}
    \label{fig:HG theta}
\end{figure}

\section{Bijective proof of Theorem~\ref{thm: HG skew SSYT} for SSYT of border strips}\label{sec: bijective proof for borderstrips}

In the previous section we showed that $\Phi^{-1}$ corresponds to the Hillman--Grassl map restricted to minimal SSYT of shape $\lm$, giving a bijective proof of a special case of Theorem~\ref{thm: HG skew SSYT}. In this section we prove bijectively another case of this theorem for SSYT whose shapes are border strips.

\begin{theorem} \label{thm: bijection border strips}
Let $\lambda/\mu$ be a border strip and $T$ be a SSYT of shape $\lambda/\mu$. Then $HG(T) \in \mathcal{A}^*_D$ for some excited diagram $D$ in $\mathcal{E}(\lambda/\mu)$.
\end{theorem}

The first step is to show that the support of the resulting array $A=HG(T)$ is in $\mathcal{A}_D$ for some excited diagram $D$ in $\mathcal{E}(\lambda/\mu)$ corresponding to some lattice path $\gamma$ from $(\lambda'_1,1)$ to $(1,
\lambda_1)$.

\begin{lemma}\label{lem:border-lattice path}
Let $\lambda/\mu$ be a border strip and $T$ be a SSYT of shape $\lambda/\mu$. Then $HG(T) \in \mathcal{A}^*_D$ for some excited diagram $D$ in $\mathcal{E}(\lambda/\mu)$, and in particular the nonzero entries of the array are contained in a SW to NE lattice path.
\end{lemma}

\begin{proof}
Since the shape of $T$ is a border strip, for an integer $k$ with  $1-\ell(\lambda) \leq k \leq \lambda_1-1$,  by Theorem~\ref{thm: Greene's theorem for HG}~(ii) we have that $dc_1(A_k)=1$. Then the nonzero entries in each rectangle $A_k$ form a NE lattice path for each $k$. Naturally, if two rectangles $A_k$ corresponding to different $k$'s overlap, then paths in each should also coincide at the intersection of the two rectangles, and so for each $k$ we get portions of one large NE lattice path.  
\end{proof}

 We define a piece-wise linear map $C$ from border strip SSYT of shape $\lambda/\mu$ to lattice paths in $[\lambda]$ with nonnegative integer entries as follows.
 Let $T$ be a border strip SSYT of shape $\lambda/\mu$, such that  $T= T' \cup T_1 \cup T_2$, where $T_2 = [a_1,\ldots,a_k]^{\top}$ is the last column of $T$, indexed $j$,  and $T_1=[b_k]$ is attached to $T_2$ horizontally on the left. Let $r=\max \{ i \,\mid\, a_k-b_k> a_{i-1}\}$ where we assume $a_0=0$. Set $p_i:= a_i - a_{i-1}$ for $i=1,\ldots,r-1$, where $a_0=0$ and set $p_r:=a_k-b_k-a_{r-1}$. 
 
 Let $R:=T'\cup [b_r,\ldots,b_k]^{\top}$, i.e the border strip SSYT comprised of $T'$ and the column attached on its right, where
 $b_i := a_i - (a_k -b_k)$ for $i=r,\ldots,k$. 
 
 Then define $C(T) := C(R) \cup [p_1,\ldots,p_r]^{\top}$ recursively as a path in $[\lambda]$ ending at $(1,j)$, where  $C(R)$ is a path in $[\lambda]$ ending at $(r,j-1)$ and the entries at $(i,j)$ for $i\geq r$ of $C(T)$ are  equal to $p_i$. See Figure~\ref{fig: alternate bijection border strips} for an example of $C(T)$.

\begin{figure}
\centering
\includegraphics[scale=0.7]{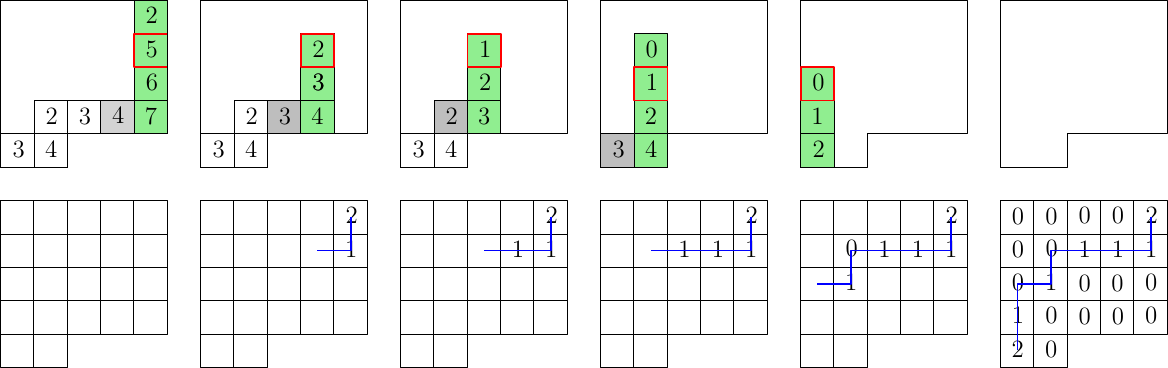}
\caption{Illustration of the map $C(T)$ from border strip SSYT to lattice paths in $[\lm]$ with nonnegative integer entries that coincides with $\HG(T)$. The cells of $T'$ are filled in green, the cell of $T_1$ is filled in {gray}, and the cell $a_r$ has a {red} border.}
\label{fig: alternate bijection border strips}
\end{figure}

\begin{proposition}
Let $\lambda/\mu$ be a border strip and $T$ be a SSYT of shape $\lambda/\mu$. Then $HG(T)$ has support in a SW to NE lattice path, such that all up steps are strictly positive entries and the entries are piece--wise linear functions of the entries of $T$. Moreover, $HG(T)=C(T)$ for $C(\cdot)$ as defined above. 
\end{proposition}

\begin{proof}
By Lemma~\ref{lem:border-lattice path} the nonzero entries of $HG(T)=A$ form a SW to NE lattice path, so we need to show the shape and entries match $C(T)$. 
Let $T$ be a border strip SSYT  as above. Let also entry $a_k$ be in the $t$-th diagonal of the shape. By Theorem~\ref{thm: Greene's theorem for HG} we have that $|A_t| =a_k$ and $|A_{t-1}| = b_k$ and $|A_{t+i}|=a_{k-i}$. Moreover, since each diagonal of $T$ is a partition of one part, say $\nu=(\nu_1)$ at diagonal $p$, then $|A_p|=ac_1(A_p) = \nu_1$. Thus, since $HG$ is a bijection, to show that $HG(T)=C(T)$, which both have supports on lattice paths, it is enough to show that the sum of entries in $A_p$ and $C(T)_p$ coincide for all diagonals $p$. 
This is now straightforward to check by the explicit form of $C(T)$ and induction on the number of columns. 

Finally, we note that in the last column of $C(T)$, the up steps of the lattice path have entries $p_i>0$ because $a_i>a_{i-1}$. Also, $R$ is a border strip SSYT since $b_i -b_{i-1} = a_i-a_{i-1}>0$ with nonzero entries up to row $r+1$, and by induction $C(R)$ is a lattice path ending in $(r,j)$ with nonzero entries on the up steps.

Thus the sum of entries in the last column of $A_t$ is equal to $a_k-b$. Similarly, the sum of entries in $A_{t+i}$ is $a_{k-i}$. Since the support of $A$ is a SW to NE lattice path, let the support in the last column extend to row $r$ and then the path turns horizontally, i.e. going west. Let the entries along the path be $p_1,\ldots,p_r, q_1,\ldots$ so that the entries in the last column are $p_1,\ldots,p_r$ with $p_r>0$. Then for $i<r$ we have
$p_1+\cdots+p_i = |A_{t+k-i}|=a_i$, so $p_i=a_i-a_{i-1}>0$. Next, the path takes a turn at $i=r$, so we have $p_1+\cdots+p_r = |A_{t}|-|A_{t-1}|=a_k-b$ and $p_r = a_k-b -a_{r-1}>0$ by the choice of $r$. The remaining path is $HG(T')$, where $T'$ is a border strip SSYT which coincides with $T$ until the last two columns. Note that by the choice of $r$ we have that $b=|A_{t-1}| \geq a_k - a_r$ since the nonzero entries in $A_t \setminus A_r$ are in $A_{t-1}$. Thus  we have that  $b\geq a_k-a_r > a_{k-1}-p_r>\ldots>a_r-p_r$  which is again a border strip SSYT and by induction we recover the other entries of $A$.
\end{proof}

\section{Additivity of Hillman--Grassl for straight shapes}\label{sec: additivity}

In this section we show the following result.

\additivityHGstraight

Before we proceed with the proof of this fact we will show a useful property of the ascending chains.

\begin{proposition}\label{prop:ac paths}
    Let $A$ be a rectangular array on $[a]\times [b]$ with nonnegative entries and $ac_t(A)$ be the maximal weight in the disjoint union of $t$ ascending chains in $A$. Then there exists a collection of disjoint connected border-strip paths of boxes $\rho_1,\ldots,\rho_t$ in $A$, such that $\rho_i$  starts at $(1,i)$ and ends at $(a,b-i+1)$, and
    $$\sum_i \sum_{p \in \rho_i} A_p = ac_t(A),$$
    i.e. the total weight of entries in this collection is maximal. 
\end{proposition}
This is saying that the maximal ascending chains can be made to start and end at the sides of the rectangle.
\begin{proof}
    First, we observe that if there are two disjoint ascending chains they can be made non-crossing (i.e. one is above the other). To see this, form monotone piecewise linear paths $\tau^1, \tau^2$ from the centers of the squares of these chains, with starting and ending points $P_1,Q_1$ and $P_2,Q_2$, respectively. Suppose that these paths cross at points  $X_1, X_2,\ldots, X_k$. We can then assign the piecewise linear paths $\tau':= \tau^1(P_1 \to X_1)\tau^2(X_1 \to X_2)\tau^1(X_2 \to X_3)\cdots$ and $\tau'':=\tau^2(P_2 \to X_1)\tau^1(X_1 \to X_2) \cdots $. These paths cover the same set of squares, and are monotonous (i.e. ascending), but now do not cross over. 

    We can now assume that the paths do not cross. We now proceed with a greedy algorithm to relabel the paths to satisfy the condition in the statement. Let $D$ be the union of squares giving a maximal $ac_t(A)$. Since the weights in $A$ are nonnegative, adding boxes to the paths in $D$ will only increase the total weight. Thus, we can assume that $D$ extends to the borders of $A$. Let $\rho_1$ be the bottom-most squares of $D$. If it does not reach the lowest row in $A$, we can add/reassign those boxes to $\rho_1$, so it contains $(1,1)$. Suppose that $(1,1),(1,2),\ldots,(1,i) \in \rho_1$ for some $i$. As the other paths are above $\rho_1$, we can extend one, say $\rho_2$ to contain $(1,i+1)$. We then assign boxes $(1,2),\ldots,(1,i)$ from $\rho_1$ to $\rho_2$, keeping it ascending and preserving the weight, and add $(1,1),(2,1),(2,2),\ldots,(2,i)$ to $\rho_1$.  If $t>2$, we can then assign boxes $(1,3),(1,4),\ldots$ to $\rho_3$, assign $(1,2),(2,2),(2,3),\ldots$ from $\rho_1$ to $\rho_2$ and add $(3,1),(3,2),\ldots$ to $\rho_1$. 
    Similarly, we can perform the same procedure on the other side of $A$ by relabeling paths and possibly assigning extra boxes to them. This procedure, greedily pulling the paths SE from the border of $D$, by possibly adding boxes, will make and keep the paths connected. 
\end{proof}

\begin{proof}[Proof of Theorem~\ref{thm:additivity straight shape}]
Let $A = HG(T)$, $M=HG(T_{min})$. The statement will follow directly from the RSK connection with HG (Theorem~\ref{thm: Greene's theorem for HG}) if we show that for every $k$ and $t$ we have
$$ac_t(A_k + M_k)=ac_t(A_k)+ac_t(M_k).$$
We have that $M_k$ consists of $1$ below the main diagonal, and 0s above. For simplicity, we will identify $M_k$ as the collection of boxes with values 1, i.e. its support.  

We now apply Proposition~\ref{prop:ac paths} to $A_k+M_k$ and $A_k$. Assume that the rectangles are $[a]\times[b]$ (so $b-a=k$) and that $a \leq b$, otherwise we can reflect the application of the Proposition, so that the paths' endpoints be $(i,1)$ and $(a-i+1,b)$. Let $D$ be the collection of $t$ ascending chains of maximal weight for $A_k$, according to Proposition~\ref{prop:ac paths} we can choose them to start at $(1,i)$ and end at $(a,b-i+1)$ and be continuous border strips. Let $D'$ be the corresponding collection of $t$ ascending chains of maximal weight for $A_k+M_k$ with the same properties. If $\rho_i'$ is a continuous monotonous path from $(1,i)$ to $(a,b-i+1)$ then $|\rho_i' \cap M_k| = b-i$ (the grid distance between $(1,i)$ and the main diagonal). Thus
$\rho'_i(A_k+M_k)=\rho'_i(A_k) + b-i$ and so
$$ac_t(A_k+M_k) = D'(A_k) + (b-1)+\cdots+(b-t),$$
where $D'(A_k)$ is the total weight of the boxes $D'$ on $A_k$. Thus, the above weight is maximal iff the total weight $D'(A_k)$ is maximal. Since $D'$ is a collection of $t$ disjoint ascending chains, then 
$$\max_{D'} D'(A_k) =ac_t(A_k).$$
Since $ac_t(M_k) = (b-1)+\cdots +(b-t)$, we thus get that
$$ac_t(A_k+M_k) = ac_t(A_k)+ac_t(M_k)$$
for every $t$ and $k$ and the claim follows.
\end{proof}

The following example shows that Theorem~\ref{thm:additivity straight shape} does not hold for skew shapes. 

 \begin{figure}
    \includegraphics[scale=0.8]{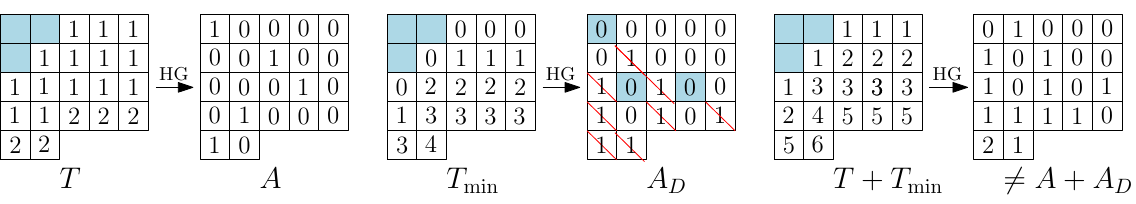}
    \caption{Counterexample for additivity of Hillman--Grassl for skew shapes.}
    \label{fig: counterex additivity skew shape}
 \end{figure}

\begin{example}
Let $T$ be the skew SSYT of shape $\lm= 55552/21$ in Figure~\ref{fig: counterex additivity skew shape} and $T_{\min}:=\Phi(D)$ be the minimal SSYT corresponding to the excited diagram $D=\{(1,1), (3,2), (3,4)\}$. As illustrated in the figure, $\HG(T) + \HG(T_{\min})\neq \HG(T+T_{\min})$.
\end{example}

\begin{figure}
    \centering
    \includegraphics[scale=0.8]{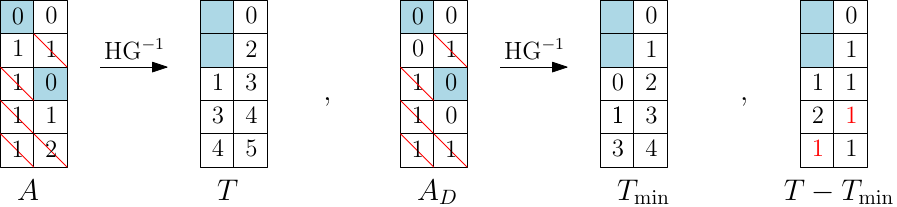}
    \caption{Counterexample for additivity of Hillman--Grassl for skew shapes.}
    \label{fig: counterex difference is rpp skew shape}
\end{figure}

For straight shapes, given a SSYT $T$ of shape $\lambda$ and $T_{\min}=T_{\min}(\lambda)$, then $T-T_{\min}(\lambda)$ is a reverse plane partition. This does no hold for skew shapes as the next example shows.

\begin{example} \label{ex: counterex difference}
Let $T$ be the skew RPP of shape $\lm= 22222/11$ in Figure~\ref{fig: counterex difference is rpp skew shape} that corresponds to the excited diagram $D=\{(1,1),(3,1)\}$. If $T_{\min}=\Phi(D)$ then $T-T_{\min}$ is not a reverse plane partition.
\end{example}

\section{Final Remarks}\label{sec: final remarks}

\subsection{A bijective proof of Theorem~\ref{thm:qNHLF skew}}

This article was in part motivated towards giving an entirely combinatorial proof of Theorem~\ref{thm:qNHLF skew} from \cite{MPP1}. From Theorem~\ref{thm: phi vs hg} and Corollary~\ref{cor:identity leading terms}, we understand combinatorially the leading terms of each summand on the right-hand-side of~\eqref{eq:qnhlf} in terms of the minimal SSYT of shape $\lm$. Also, in Section~\ref{sec: bijective proof for borderstrips} we have a bijective proof for border strips $\lm$. It would be interesting to understand combinatorially other cases like shapes with Lascoux--Pragacz decomposition $(\theta_1,\theta_2)$ or {\em thick zig-zags} $\delta_{n+k}/\delta_n$. 

\subsection{Comparison with the formula for skew tableaux involving Littlewood--Richardson coefficients}

In Section~\ref{sec: comparison OOF} we compared the number of terms of the positive formulas \eqref{eq:nhlf} and \eqref{eq:OO} for $f^{\lm}$ and showed that $\nED(\lm)\leq \nOOT(\lm)$.
Another positive formula for $f^{\la/\mu}$ is given using the {\em Littlewood-Richardson coefficients} $c^{\la}_{\mu,\nu}$, namely
\begin{equation} \label{eq:LR}
f^{\la/\mu} \,=\, \sum_{\nu \vdash |\lm|} c^{\la}_{\mu\nu} f^\nu. 
\end{equation}
The terms $f^{\nu}$ are readily computed by the classical hook-length formula \eqref{eq:hlf}. It is natural to compare this formula to \eqref{eq:nhlf} and \eqref{eq:OO} and see which one is more ``efficient'', in the sense of number of summands. Since the $c^{\la}_{\mu\nu}$ do not have explicit product formulas we will regard them as multiplicities and thus the question is to compare
$$ \nED(\la/\mu), \, \nOOT(\lm) \qquad \text{vs} \qquad \LR(\la/\mu):= \sum_{\nu \vdash n} c^{\la}_{\mu\nu}.$$

From one of the rules to compute $c^{\lambda}_{\mu,\nu}$ (see \cite[Ch. 7, A1.3]{EC2}), $\LR(\lm)$ counts the number of {\em Littlewood--Richardson tableaux} : which are SSYT $T$ of skew shape $\lm$ with positive entries whose {\em reverse reading word} (reading the entries of $T$ row by row right-to-left and bottom-to-top)  is a {\em lattice permutation}, i.e. in every initial factor of the word the number of $i$s is at least as many as the number of $(i+1)$s, for each $i$. Denote the set of such tableaux by $\mathcal{LR}(\lm)$. From this interpretation of $\LR(\lm)=|\mathcal{LR}(\lm)|$ we obtain the following inequality.

\begin{proposition}
For a skew shape $\lm$, we have that $\LR(\lm)\leq \nOOT(\lm)$.    
\end{proposition}

\begin{proof}
We show that $\mathcal{LR}(\lm)\subseteq \mathcal{SF}(\lm)$ which implies the result since by Proposition~\ref{bij:OOT and flagged skew tableaux}, we have that $|\mathcal{SF}(\lm)|=\nOOT(\lm)$.

Given a tableau $T$ in $\mathcal{LR}(\lm)$, by the lattice condition of its reverse reading word $w$, the entries in the first row of $T$ are all $1$s. Iteratively, if the entries of $T$ in row $i$ are at most $i$, then the word $w$ up until that point has entries at most $i$. Then the last entry of row $i+1$, which comes next in $w$ is at most $i$ by the lattice condition. Thus the entries of $T$ in row $i+1$ are at most $i$ and so $T\in \mathcal{SF}(\lm)$, as desired.
\end{proof}

\begin{example}
A quick computation for $\la/\mu =4321/21$ shows that $\nED(\la/\mu) = 5$, $\LR(\la/\mu) = 11$, and $\nOOT(\lm)=17$. On the other hand when $\la/\mu = 43/2$ we have $\nED(\la/\mu) =\nOOT(\lm)=3$ and $\LR(\la/\mu)=2$. Also when $\lm=8642/21$ we have $\nED(\lm)=8$ and $\LR(\lm)=\nOOT(\lm)=20$.    
\end{example} 

From these examples, the question arises for which shapes $\lm$ are: 
\[
\LR \leq \nED\leq \nOOT, \quad \text{ or } \quad   \nED\leq \LR\leq \nOOT?\]
In the first example above we see that the minimal SSYT tableaux are actually Littlewood--Richardson tableaux. This also raises the question of which shapes $\lm$ give equalities $\nED(\la/\mu)=\LR(\la/\mu)$ and $\LR(\lm)=\nOOT(\lm)$, respectively?

\subsection{Multivariate identity from the Hillman--Grassl bijection}

In \cite[Thm. 1.2]{NO}, Naurse--Okada gave a multivariate identity for skew RPP. Because the Hillman--Grassl bijection is behind the $q$-analogue \eqref{eq:qnhlf}, we can derive an analogous multivariate identity of skew SSYT. We need some notation. The variables are $\z = (z_i)_{i\in I}$, where $I\subset \mathbb{Z}$ is the set of contents of $\lam$, that is the set of integers $\{1-\ell(\lam),\dots, \lam_1-1 \}$. For $\pi \in \RPP(\lm)$, we let 
\[
\z^{\pi}:= \prod_{u \in [\lm] }z_{c(u)}^{\pi(u)}, \qquad 
\z[H_{\lam}(u)] := \prod_{v \in H_{\lambda}(u)}z_{c(v)},
\]
where $H_{\lambda}(u)$ are the cells of $[\lm]$ in the hook of cell $u\in [\lm]$. Combining Theorem~\ref{thm: HG skew SSYT} and Proposition~\ref{prop:trace and HG} gives the following identity.

\begin{theorem}
Let $\lm$ be a skew shape. We have the following multivariate identity:
\begin{equation*}
 \sum_{T \in \SSYT(\lambda/\mu)}  \z^{T} \,=\, \sum_{D \in \E(\lm)} \frac{\prod_{u \in \Br(D)}\z[H_{\lam}(u)]}{\prod_{u\in D} (1-\z[H_{\lam}(u)])}.
\end{equation*}
\end{theorem}

We end with a multivariate analogue of Corollary~\ref{cor:identity leading terms} obtained by combining Theorem~\ref{thm: phi vs hg} and Proposition~\ref{prop:trace and HG}.

\begin{corollary}
  Let $\lm$ be a skew shape. We have the following multivariate identity:
\begin{equation*}
 \sum_{T \in \SSYT_{\min}(\lambda/\mu)}  \z^{T} \,=\, \sum_{D \in \E(\lm)} \prod_{u \in \Br(D)}\z[H_{\lam}(u)].
\end{equation*}
  
\end{corollary}

\noindent {\bf Acknowledgements:} We thank Igor Pak for having introduced the first two named authors with Naruse's formula and for the ensuing collaboration studying this formula. We also thank Peter Cassels, Ang\`ele Foley, Swee Hong Chan, Soichi Okada, and Daniel Zhu for helpful conversations. This work was facilitated by computer experiments using Sage \cite{sagemath}, its algebraic combinatorics features developed by the Sage-Combinat community \cite{Sage-Combinat},

\begin{figure}
    \centering
\includegraphics[width=0.9\textwidth]{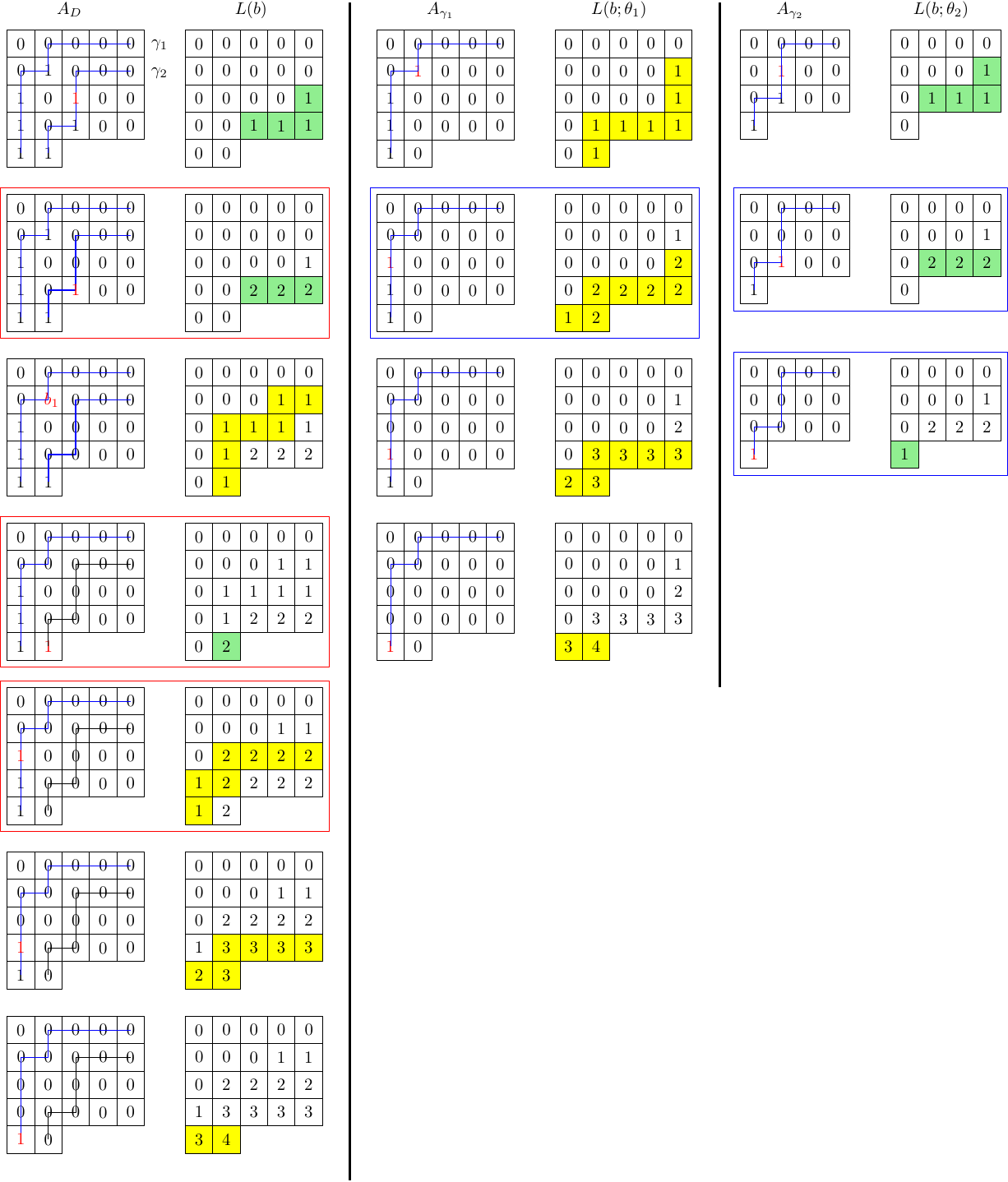}
    \caption{
    The first two columns show the inverse 
    Hillman-Grassl paths $L(b)$ on $A_{D}$ for a diagram $D$ with Kreiman decomposition $(\gamma_1,\gamma_2)$. The next columns show the paths $L(b;\theta_1)$ on $A_{\gamma_1}$ and the paths $L(b;\theta_2)$ on $A_{\gamma_2}$ respectively. By Corollary~\ref{lem: additivity of k paths}, the multiset $\bigcup_{b\in \Br(D)} L(b)$ equals the union of the multisets $\bigcup_{b\in \Br(\gamma_i)} L(b;\theta_i)$ for $i=1,2$.}
    \label{fig:HG combined and separate}
\end{figure}

\end{document}